%% file: btf5.tex
\documentclass[11pt,a4paper]{article}

\title{Backward theory supports modelling via invariant manifolds for non-autonomous dynamical systems}

\author{A.~J. Roberts\thanks{School of Mathematical Sciences, University of Adelaide, South Australia 5005.
\protect\url{mailto:profajroberts@protonmail.com}
\quad Also \protect\url{http://orcid.org/0000-0001-8930-1552}}}

\date{31 May 2022}

\IfFileExists{ajr.sty}{\usepackage{ajr}}{}
\usepackage{amsmath,amsthm,amsfonts,defns}
\usepackage{natbib}
\AtEndDocument{\raggedright\bibliography{bibexport}}

\newenvironment{full}{}{}

\usepackage[capitalise,nameinlink,noabbrev]{cleveref}
\crefname{equation}{}{} 
\Crefname{assumption}{Assumption}{Assumptions} 
\let\eqref\cref

\usepackage{pgfplots}
\pgfplotsset{compat=newest}
\usepgfplotslibrary{external}
\allowdisplaybreaks
\renewcommand{\vec}[1]{\text{\boldmath$#1$}}
\newcommand{\Z}[1]{e^{#1t}{\star}}
\newcommand{\sgn}{\operatorname{sgn}}
\newcommand{\res}{\operatorname{Res}}

\newcommand{\Orc}[1]{\ensuremath{\mathbb O(#1)}}
\Cal A\Cal C\Cal L\Cal E\Cal F\Cal G\Cal M\Cal N\Cal P\Cal V
\Frak y \Frak f \Frak g\Frak h \Frak x \Frak p \Frak F
\Vec a\Vec b\Vec c\Vec f\Vec g\Vec h\Vec p\Vec q\Vec r
\Vec u\Vec x\Vec y\Vec z
\Vec F\Vec G\Vec H\Vec U\Vec X\Vec Y\Vec Z
\def\ov{\ensuremath{\vec 0}}
\Vec{epsilon}
\Bb D\Bb R\Bb M\Bb T
\def\cond{\operatorname{cond}}
\def\diag{\operatorname{diag}}
\def\tr#1{#1^{\sf T}}

\def\ou\big(#1,#2,#3\big)%
    {{{\rm e}^{\if#31\else#3\fi t}\star}#1\,}

\newcommand{\Lip}{\operatorname{Lip}}

\newcommand{\muc}{\ensuremath{\tilde\mu}}

\numberwithin{equation}{section}

\RaisedNamesfalse

\begin{document}

\maketitle


\begin{abstract}
In the area of invariant\slash integral manifolds for non-autonomous dynamical systems, this article establishes the foundation for a theory that complements extant theory. 
Current rigorous support for dimensional reduction modelling of slow-fast systems is limited by rare events in stochastic systems that may cause escape, and limited in many applications by restrictions on the nature of \pde\ operators.
To circumvent such limitations, we initiate developing a backward theory of invariant\slash integral manifolds that complements extant  theory. 
Here, for deterministic non-autonomous \ode\ systems, we construct a conjugacy with a normal form system to establish the existence, emergence and exact construction of center manifolds in a finite domain for systems `arbitrarily close' to that specified.
A benefit is that the constructed invariant manifolds are known to be exact for systems `close' to the one specified, and hence the only error is in determining how close over the domain of interest for any specific application.
Built on the base developed here, planned future research should develop a theory for stochastic and/or \pde\ systems that is useful in a wide range of modelling applications.
\end{abstract}

%

\begin{full}
\tableofcontents
\end{full}

\section{Introduction}


Centre manifold theory provides an excellent framework to justify and construct reduced low-dimensional models of some specified high-dimensional dynamical system \cite[e.g.,][]{Roberts2014a}, complete with correct treatment of initial conditions, forcing and uncertainty \cite[e.g.,][]{Roberts89b}, and giving correct boundary conditions for homogenised\slash slowly-varying \pde\ models \cite[e.g.,][]{Roberts92c}.
Here we establish a first step in complementing extant theory with a new backward theory applicable to quite general non-autonomous systems and to systems that have center, stable and unstable dynamics---other backward theory has proved very useful in other contexts~\cite[e.g.,][]{Grcar2011, Ghosh2018}.
Indeed, \cite{Corless2019} comment that backward error analysis ``\textsc{bea} has become a general method fruitfully applied to problems \ldots\ This is hardly a surprise when one considers that \textsc{bea} offers several interesting advantages over a purely forward-error approach''.
A further step in this project is the start on corresponding backward theory for \(\infty\)-D systems recently established by \cite{Hochs2019}.
But for the more accessible finite-D systems addressed herein we present arguments extended to non-self-adjoint systems (\cref{sec:stsenfsmsde}), and provide an unprecedented lower bound on the finite domain of validity (\cref{sec:nfeep}).


The theory developed complements existing theory by addressing the issues from a different direction and thereby empowering new views.
For example, whereas \cite{Bates2008} [p.1] prove that ``if the given manifold is approximately invariant and approximately normally hyperbolic, then the dynamical system has a true invariant manifold nearby'' \cite[also][]{Henry81, Aulbach82}.  
Such theory is called a \emph{forward theory} because it proves for a given system there is a result, here an invariant manifold, that is approximated to some reasonable error. 
In contrast, a \emph{backwards theory}, such as that establiashed in\cref{sec:nfeep}, proves that ``if a constructed manifold is \ldots, then it is an exact invariant manifold for a dynamical system which approximates the given system to some reasonable error''.
\cite{Grcar2011} usefully overviews, discusses, and compares familiar forward and backward theorems \text{in linear algebra.}

Of course, forward theory will continue to be important in its rigorous results about many dynamical systems.
The aim of this research program is to \emph{complement} forward theory: to complement it by a backwards theory that establishes slightly different results for a different range of scenarios, and hence results which certify practically useful properties for a wider range of physical applications.
The backward theory here is to encompass center manifolds of more than just a few dimensions.
Consequently numerical manifold continuation methods \cite[e.g.,][]{auto01, England2007, Ziessler2018}, albeit valuable for some low-D manifolds, are far from practical over the entire \text{scope addressed here.}

There is a further pronounced difference in the complementary approach developed here.
We invoke Lagrange's remainder theorem to put lower bounds on the size of the domain of validity (\cref{lem:nfftd}). 
Such general finite bounds have not been identified before because most approaches invoke `$\epsilon\to0$' or `\(\epsilon\)-neighbourhood' theorems.
The only cognate results are by \cite{Lamarque2011} who developed upper bounds on the domain of validity for normal forms of  nonlinear normal modes in finite-D autonomous systems, and by \cite{Hooton2020} who estimate the length of branches of periodic orbits in a {Hopf} bifurcation.
In essence, we invoke a slightly extended normal form coordinate transform (\cref{sec:stsenfsmsde}) and use it from a new point of view.
In the context of normal forms, \cite{Lamarque2011} commented 
\begin{quote}
The main recognised drawback of perturbation methods is the absence of a criterion establishing their range of validity in terms of amplitude.
\end{quote} 
Here the backward approach provides new useful results on the range of validity---the lower bound of \cref{lem:nfftd}.
\cref{cor:stbackex,pro:stbackem} summarise the resultant theory including the finite size of the range.

\begin{full}
The normal form coordinate transformation of this backward approach also illuminates apparent paradoxes.
\cref{sec:asfss} discusses, through a key example \cite[]{Lorenz86, Lorenz87}, how the question of whether a slow manifold exists or not in geophysical fluid dynamics may be answered by arguing that for every given slow-fast system in the relevant class, there is an arbitrarily close system that possesses, in a finite domain, a slow manifold of interest.
This existence of a nearby system with a slow manifold is irrespective of whether or not the given system has a slow manifold.
This example is just one case where the original system does not have useful invariant manifolds, but the constructed approximate systems do possess such invariant manifolds, \emph{and} where the invariant manifolds are vital guiding\slash organising concept in the application \cite[e.g.,][]{Leith80}. 
\end{full}


\begin{example}\label{egUdwadia2022}
\cite{Udwadia2022} explored how to control a nonhyperbolic system---the discussion may also be viewed as a backwards approach.
The \emph{Numerical Example~1} by \cite{Udwadia2022} is the system
\begin{equation}
\dot x=xy+\tfrac12x^3+xy^2,
\quad
\dot y=-y-2x^2+x^2y\,.
\label{eqUdwadia1}
\end{equation}
One may straightforwardly construct \cite[e.g.,][]{Carr81} that the system~\eqref{eqUdwadia1} has an approximate slow manifold \(y=-2x^2+\Ord{x^4}\) on which the approximate evolution is \(\dot x=-\tfrac32 x^3+\Ord{x^5}\).
\cite{Udwadia2022} then discussed that one could apply a \emph{small} control~\(w:=8x^4(1-2x^2)\) to the system~\eqref{eqUdwadia1}---remarkably applied to the slaved mode~\(y\) not the principal mode~\(x\)---namely
\begin{equation}
\dot x=xy+\tfrac12x^3+xy^2,
\quad
\dot y=-y-2x^2+x^2y+w,
\label{eqUdwadia2}
\end{equation}
Then the `closed-loop' controlled system~\eqref{eqUdwadia2} has an exact slow manifold \(y=-2x^2\) with the exact evolution \(\dot x=-\tfrac32x^3+4x^5\).
\cite{Udwadia2022} comments that such exact results are useful in the reliable control of systems.
Equivalently, this is an example of the backwards approach developed herein: for a given system, such as~\eqref{eqUdwadia1}, our backwards  approach constructs a nearby system, such as~\eqref{eqUdwadia2}, which possesses an exact invariant manifold, such as \(y=-2x^2\), with its exactly known evolution. 
\end{example}

Akin to the comments by  \cite{Lamarque2011}, \cite{Udwadia2022} goes on (p.75) to comment that ``Determining the basin of attraction \ldots\ is generally difficult''.   Here, \cref{lem:nfftd} is an unprecedented lower bound on the size of the basin of attraction, and contributes to answering this identified general difficulty.

\paragraph{Finite time theory for non-autonomous systems}
The topic of finite time phenomena and associated invariant sets and manifolds are increasingly important in understanding coherent structures, mass transport, and metastability in non-autonomous systems \cite[e.g.,][]{Haller2000, Froyland2010, Froyland2013b, Balasuriya2016}.
A crucial reformation to empowering backward theory to usefully address such non-autonomous systems is to modify the definition of key invariant manifolds (\cref{def:nfnfim}).
The classic definition of un\slash stable and center manifolds requires the existence of solutions and their limits as time goes to~\(\pm\infty\) \cite[e.g.,][]{Henry81, Potzsche2006, Haragus2011, Barreira2007, Aulbach2006}.
This in turn requires solutions of the dynamical system to be reasonably well-behaved for all time, which places strong constraints on the systems that can be studied---constraints that in applications are often not available, or are hard to establish.
For example, in stochastic systems very rare events will eventually happen in infinite time, requiring global Lipschitz and boundedness that are oppressive in applications.
By modifying the classic definitions we establish results for finite times, which are useful in many applications, and for a usefully wide range of systems, as indicated by \cite{Wiggins2005} [p.314]:
\begin{quote}
finite time, this poses some severe problems with applying dynamical systems-type ideas because dynamical systems theory is often described as the study of the `long time behavior' of a system. 
The mathematical definitions of hyperbolic trajectories, stable and unstable manifolds of hyperbolic trajectories, KAM tori, and chaos are inherently `infinite time' in nature.  
If we have a flow field that is only known for a finite time, and is aperiodic in time so that no inference about the behavior outside the known time interval can be made, how can we possibly proceed with a dynamical systems analysis? 
This is an area \text{of continuing research,}
\end{quote}
This article makes a significant contribution towards resolving such ``severe problems'' of ``dynamical systems analysis''.



The problems indicated by \cite{Wiggins2005} are ongoing.
For example, \cite{Duc2008} commented that ``motivated by applications in fluid mechanics \ldots\ One of the central questions is whether hyperbolicity and invariant manifolds have any useful meaning for a finite time.''
And recently, \cite{Newman2021} 
\begin{quote}
present a stabilization phenomenon 
\ldots
the traditional approaches to theoretical stability analysis such as 
\ldots
infinite-time approaches cannot generally be applied to describe the stabilization phenomenon
\ldots
the phenomenon should be understood in terms of finite-time stability,
\ldots
whose necessity in the study of various nonlinear systems, such as climate and biological systems, is increasingly being recognized. 
\end{quote}
The novel backwards theory established in \cref{sec:nfeep} encompasses finite-time analysis of non-autonomous dynamical systems---to complement extant forward theory which dominantly requires infinite-time.

\subsection{General scenario}
In this article we consider dynamical systems for dependent variables $\uv (t)\in\RR^{m+n+\ell}$ in the class
\begin{equation}
\dot\uv=\cL\uv+\cN(t,\uv)
\label{eq:phys}
\end{equation}
for linear operator~\cL\ and strictly nonlinear \(\cC^\fp\)-function~\cN.
Of course there is a considerable body of forward theory that applies to finite-D non-autonomous systems in this form \cite[e.g.,][]{Henry81, Knobloch82, Potzsche2006, Haragus2011}---the import of reconsidering this class of systems is as the starting point for developing a new backward theory.
Interestingly, while this article was under review, \cite{WenleiLi2022} developed a backwards theorem for center manifolds (their Theorem~5.1) but restricted only to existence, only to zero real-part eigenvalues, only to autonomous centre-stable systems, and only in some neighbourhood of the origin.
The backwards theory established in \cref{sec:nfeep} gives many extra significant results: an example being the new lower bound on the size of the applicable domain (\cref{lem:nfftd}).

In applications, a coordinate independent approach, developed from that by \cite{Coullet83}, provides an efficient direct asymptotic construction of specific chosen invariant manifolds for systems in the form~\cref{eq:phys} \cite[Ch.~5, e.g.]{Roberts2014a}.
The procedures of \cref{sec:stsenfsmsde} are not intended to be a practical construction, they are intended to be a tool for developing the \text{theory of \cref{sec:nfeep}.}

For the purposes of developing the theory, let's assume a preliminary linear change of basis separates the physical system~\cref{eq:phys} into variables $\xv(t)\in\RR^m$, $\yv(t)\in\RR^n$, and~$\zv(t)\in\RR^\ell$, whence the system becomes
\begin{subequations}\label{eqs:stsesde}%
\begin{align}
    \dot{\xv}&=A\xv+\fv(t,\xv,\yv,\zv),
    \label{eq:stsesdex}\\
    \dot{\yv}&=B\yv+\gv(t,\xv,\yv,\zv),
    \label{eq:stsesdey}\\
    \dot{\zv}&=C\zv+\hv(t,\xv,\yv,\zv),
    \label{eq:stsesdez}
\end{align}
\end{subequations}
and where \cref{ass:givensys} details conditions to be satisfied by matrices~\(A,B,C\) and nonlinear functions~\(\fv,\gv,\hv\).
The conditions identify that, about the origin, \xv~are center variables, \yv~are fast stable variables, and \zv~are fast unstable variables (\cref{def:ocf}).


Recall that in applications of mathematics one never knows the precise equations that govern a given physical system.
As \cite{Feynman98} [p.2] wrote: ``everything we know is only some kind of approximation, because \emph{we know that we do not know all the laws} as yet.''
Consequently, a differential equation system that is `close enough' to the originally specified mathematical system~\cref{eqs:stsesde} may be as equally a valid description of the modelled physical system.
Hence, this article focusses on establishing exact properties for systems `close' to the original~\cref{eqs:stsesde}  (akin to \cref{egUdwadia2022}).
The following \cref{cor:stbackex} is a first such backwards property.
Such backward propositions are useful in other mathematical areas \cite[e.g.,][]{Grcar2011}.

The following two new propositions of practical importance are established via \cref{thm:stsenfp,thm:nfimex,lem:nfftd,cor:nfeom,lem:nfemerge}.

\begin{subequations}\label{eqs:stbackex}
\begin{proposition}[existence] \label{cor:stbackex}
For every non-autonomous system~\cref{eqs:stsesde} satisfying \cref{ass:givensys} there exists nearby systems (such as \cref{eq:stsexform}+\cref{eqs:stsesdenf}), asymptotically close to any specified order (\cref{def:ocf}) limited by the spectral gaps (between the eigenvalues of~\(A\) and of~\(B\) and~\(C\), \cref{def:nftpd}) and the smoothness of~\fv, \gv\ and~\hv\ (\cref{thm:stsenfp}), that in a finite domain \((t,\Xv)\in\TT_\mu\otimes d_\mu\) (\cref{def:nfnfim}) possesses a smooth non-autonomous center manifold~\(\cM_c\) (\cref{def:nfnfim}) parametrised as%
\footnote{Throughout this paper, the context distinguishes between~\(\xv(t,\Xv)\) of~\cref{eq:stbexsm} in a parametrisation of a center manifold, and~\(\xv(t)\) in a solution to the dynamical system~\cref{eqs:stsesde}, and similarly for~\yv\ and~\zv.}
\begin{equation}
\xv=\xv(t,\Xv),\quad 
\yv=\yv(t,\Xv) \quad\text{and}\quad 
\zv=\zv(t,\Xv)
\label{eq:stbexsm}
\end{equation}
where \(\xv(t,\Xv)=\Xv+\Ord{|\Xv|^2}\) is a near identity (\cref{thm:nfimex}).
On~\(\cM_c\) the evolution is of the form (\cref{cor:nfeom})
\begin{equation}
    \dot{\Xv}= A\Xv+\Fv_c (t,\Xv),
\label{eq:stbexsme}
\end{equation}
for some smooth nonlinear~\(\Fv_c\).
\end{proposition}

The next proposition characterises the exponentially quick decay towards such a centre manifold---its emergence from many initial conditions.

\begin{proposition}[emergence] \label{pro:stbackem}
For every non-autonomous system~\cref{eqs:stsesde} satisfying \cref{ass:givensys} and for every chosen rate~\(\mu\) in the spectral gap, a non-autonomous center manifold of a nearby system of \cref{cor:stbackex} attracts at the rate~\(\mu\) all of its solutions in a finite domain~\(d_\mu\) of its center-stable manifold~\(\cM_{cs}\) (\cref{lem:nfemerge}). 
That is, for all solutions~\(\big(\xv(t),\yv(t),\zv(t)\big)\) (of~\cref{eq:stsexform}+\cref{eqs:stsesdenf}) that lie in~\(\cM_{cs}\) there exists a solution~\(\Xv(t)\) to~\cref{eq:stbexsme} and a constant~\(C\) (given quantitatively by~\cref{eq:bigC}) such that
\begin{equation}
\left|\big(\xv(t),\yv(t),\zv(t)\big)
-\big(\xv(t,\Xv(t)),\yv(t,\Xv(t)),\zv(t,\Xv(t))\big)\right|
\leq Ce^{-\mu t}
\label{eq:stbackem}
\end{equation} 
for times \(t_0\leq t<T_\mu\), a time interval \([t_0,T_\mu)\subseteq \TT_\mu\) for which both solutions \(\big(\xv(t),\yv(t),\zv(t)\big)\) and \(\big(\xv(t,\Xv(t)),\yv(t,\Xv(t)),\zv(t,\Xv(t))\big)\) remain in~\(d_\mu\).
\end{proposition}
\end{subequations}

These two propositions are important in applications: 
\emph{all constructed non-autonomous center manifolds are the attractive \emph{exact} center manifolds, in a finite domain, for a non-autonomous system `close' to the one specified.}
Of course the emergence only follows from the exponential attraction of~\eqref{eq:stbackem} when the time-scale, \((T_\mu-t_0)\), for staying in the domain~\(d_\mu\) is long enough for the decay at rate~\(\mu\) to overcome, to some useful extent depending upon the context, the short-term transients following time~\(t_0\).

In applications, the finite domain~\(d_\mu\) is significant.  
This backwards approach, via \cref{lem:nfftd}, takes a first step in quantifying the size of this domain.
Even in the simpler case of hyperbolic dynamics, \cite{Lan2013} comment that ``it is difficult to state the precise region of validity of this mapping \ldots\ these theorems only provide a much under-estimated linearization region''. 
\cref{lem:nfftd} establishes a generic lower bound on the finite size of the domain for a normal form system~\cref{eqs:stsesdenf}, and hence for a center manifold~\cref{eqs:stbackex}.
Extant forward theory has no such general lower bound.

In these propositions, one chooses~\(\mu\) in the spectral gap to suit the time-scale of interest in a specific application.
The time interval~\(\TT_\mu\) would usually be determined `on-the-fly' depending upon specific initial conditions.
Lastly, usually the specific value of the constant~\(C\) in~\eqref{eq:stbackem} is largely irrelevant \text{to an application.}

What is innovative herein is \emph{not} the algebraic machinations of \cref{sec:stsenfsmsde} that constructs a useful the normal form---that developed herein is a new synthesis and extension of earlier work \cite[e.g.,][]{Knobloch82, Arnold03, Roberts06k,  Potzsche2006}.
What is innovative herein is the framework established in \cref{sec:nfeep}: the new definitions of invariant manifolds; the ability to handle more general nonlinearities; the flexibility to address geophysical\slash engineering slow manifolds;  the finite time validity; an unprecedented lower bound on the size of the domain of validity; and the new backward interpretation of the algebraic construction.

\begin{full}
\subsection{Example: an exact coordinate transformation}
\label{eg:nfect}

A key part of the approach is to establish (\cref{sec:stsenfsmsde}) a smooth conjugacy between dynamics of the original variables and the dynamics of a `normal form' system with some known exact properties.
This subsection introduces such a conjugacy for an example 2D autonomous dynamical system \cite[]{Roberts85b}.  The conjugacy shows the finite domain of the emergent 1D, slow manifold.
In time~\(t\) and variables~\(x(t)\) and~\(y(t)\) the system is
\begin{equation}
	\dot x=-xy \qtq{and} 
	\dot y=-y+x^2-2y^2,
	\label{eq:cmect}
\end{equation}
where overdots denote time derivatives.
\cref{fig:ctECTtraj}(left) plots some trajectories illustrating the exponentially quick attraction of the parabolic slow manifold \(y= x^2\).
\begin{figure}
\centering
\begin{tabular}{@{}c@{}c@{}}
\input{ctECTtrajx.ltx}&
\input{ctECTtrajXX.ltx}
\end{tabular}
\caption{selected trajectories of (left) the system~\cref{eq:cmect} in the \(xy\)-plane, and (right) the transformed system~\cref{eq:cmECT} in the \(XY\)-plane. }
\label{fig:ctECTtraj}
\end{figure}%
Marvellously, the near-identity coordinate transform
\begin{equation}
x=\frac{X}{\sqrt{1-2Y/(1+2X^2)}}
\quad\text{and}\quad
y=X^2+\frac{Y}{{1-2Y/(1+2X^2)}},
	\label{eq:ctECT}
\end{equation}
transforms the \ode{}s~\cref{eq:cmect} into the normal form system
\begin{equation}
\dot X=-X^3
\quad\text{and}\quad
\dot Y=-Y\left[\frac1{1+2X^2}+4X^2\right].
	\label{eq:cmECT}
\end{equation}
Generally, throughout this article lowercase letters denote original variables, and uppercase letters denote conjugate normal form variables.
\cref{fig:ctECTtraj}(right) plots the corresponding trajectories of~\cref{eq:cmECT}.  
The normal form system~\cref{eq:cmECT} immediately shows that \(Y=0\) is an invariant manifold.
The normal form~\cref{eq:cmECT} also shows that \(Y=0\) is exponentially attractive, for all~\(X\), at a rate of one or more.
Since \(\dot X=-X^3\) has no linear term, \(Y=0\) is called a slow manifold.
Thus the slow variable~\(X\) describes the emergent long-time evolution of~\eqref{eq:cmECT} for \emph{all} initial conditions in the \(XY\)-plane.

\begin{figure}
\centering
\begin{tabular}{@{}cc@{}}
\input{ctECT.ltx}&
\input{ctECTx.ltx}
\end{tabular}
\caption{domains of validity of the coordinate transform~\cref{eq:ctECT} in both the \(xy\)-plane (left) and the \(XY\)-plane (right).
Superimposed are selected coordinate curves at spacing of (left) \(\Delta X=\Delta Y=0.2\), and (right) \(\Delta x=\Delta y=0.2\); the black curves are transformed coordinate axes \(XY=0\) and \(xy=0\)\,.}
\label{fig:ctECT}
\end{figure}%
The exact coordinate transformation~\cref{eq:ctECT} then determines an \(xy\)-domain in which exactly corresponding statements hold for the given \(xy\)-system~\cref{eq:cmect}.
\begin{itemize}
\item In the \(XY\)-plane the \ode\ system~\cref{eq:cmECT} has no singularities and so the statements of the previous paragraph are globally valid in~\((X,Y)\).

However, the coordinate transform~\cref{eq:ctECT} is singular when the denominator \(1+2X^2-2Y=0\)\,, that is, it is singular when \(Y=\frac12+X^2\).
The domain of validity must be restricted to \(Y<\frac12+X^2\) (to include the origin) as illustrated in the right-hand plot of \cref{fig:ctECT}.

Further straightforward algebra derives that the Jacobian of the coordinate transform~\cref{eq:ctECT} has determinant
\({(1+2X^2)^{3/2}}/{(1+2X^2-2Y)^{5/2}}\)
which is never zero and so the coordinate transform only degenerates  on the singular curve \(Y=\frac12+X^2\).

\item Two limits determine the domain of validity in the \(xy\)-plane.
Firstly, as \(Y\to-\infty\), \(y\to X^2-\frac12\) so the domain is restricted to being in \(y>-\frac12\)\,, and a little more algebra shows we need \(y>-\frac12+x^2\), as illustrated in the left-hand plot of \cref{fig:ctECT}. 
Secondly, as \(Y\to\frac12+X^2\) from below, \(y\to+\infty\) and so the \(xy\)-domain is unrestricted above.
\end{itemize}

Hence a slow manifold of the \(xy\)-system~\cref{eq:cmect} is exactly \(y=x^2\) (since~\cref{eq:ctECT} on the slow manifold \(Y=0\) reduces to \(x=X\) and \(y=X^2\)) and exists globally in the \(xy\)-plane.
Further, \emph{all} initial conditions in \(y>\tfrac12+x^2\) generate trajectories which are exponentially quickly attracted to solutions on this slow manifold.

\cref{fig:ctECT} reaffirms that the coordinate transform which nonlinearly separates slow and stable variables may exist over a large domain in the \(xy\)-plane: its existence need not be restricted to a `small' neighbourhood of the origin.
This example illustrates the value of such a `normal form' coordinate transformation between conjugate dynamical systems.
Our challenge is to establish a widely useful approach to such properties for general scenarios.

\end{full}

\subsection{An example nonlinear non-autonomous system}
\label{sec:stseeg}

This example illustrates results of prime interest in applications, namely \cref{cor:stbackex,pro:stbackem,def:nfnfim}, for the dynamics of nonlinear, multi-scale, non-autonomous systems.
For some specified \idx{forcing}~$w(t)$ with strength~$\sigma$, consider the dynamics of~$(x(t),y(t))$ according to the coupled \ode{}s
\begin{equation}
    \dot x=-xy
    \quad\text{and}\quad
    \dot y=-y+x^2-2y^2+\sigma w(t).
    \label{eq:stsetoy}
\end{equation}
\begin{figure}
    \centering
    \begin{tabular}{@{}ll@{}}
    \parbox[t]{0.28\textwidth}{%
    \caption{\label{fig:stsesim1}%
    trajectories of the example system~\eqref{eq:stsetoy} from various initial conditions with \(|y(0)|=0.45\), for the forcing $w(t)=\cos t$ of strength $\sigma=0.25$\,.
The trajectories are attracted to the centre manifold $y\approx x^2+\sigma(\cos t+\sin t)/2$\,.}}
&    \raisebox{-\height}{\input{sim1.ltx}}
\end{tabular}
\end{figure}%
\cref{fig:stsesim1} plots some representative trajectories of the \ode\ system~\eqref{eq:stsetoy} for the trigonometric forcing~\(w(t)=\cos t\)\,.
In the plotted \idx{finite domain} near the origin, the $y$~variable decays exponentially quickly to oscillations about $y\approx x^2$; whereas the $x$~variable evolves relatively slowly over long times, albeit affected by the rapidly oscillating~$y$---a rapid oscillation that also intrudes on the supposedly slow~\(x\).
\cref{fig:stsesim1} plots a case of trigonometric forcing, but we discuss quite general forcing~\(w(t)\).

For the simple system~\eqref{eq:stsetoy}, well-known extant forward theory readily applies to guarantee the existence, emergence, and approximability of a \emph{local} time-dependent (non-unique) centre manifold for the specified system~\eqref{eq:stsetoy} \cite[e.g.,][]{Knobloch82, Potzsche2006, Haragus2011}.  
Standard algebraic machinations \cite[e.g.,][]{Chao95} derive that the centre manifold may be parametrised \emph{in some neighbourhood of the origin} by an evolving variable~\(X(t)\) as, \text{for example},
\begin{subequations}\label{eqs:stsetoycm}%
\begin{align}
x&= X
+\sigma   \ou\big(w,tt,-\big) X
+\Ord{X^4+\sigma^2},
\\
y&=  X^{2}
+\sigma  \ou\big(w,tt,-\big)
+2\sigma   \big(  1
- \ou\big({},tt,-\big) \big)\ou\big(w,tt,-\big) X^{2}
+\Ord{X^4+\sigma^2},
\\
\dot X&= - X^{3} 
-\sigma   w X 
+\Ord{X^4+\sigma^2},
\end{align}
\end{subequations}
where the given asymptotic errors are ``as \((X,\sigma)\to0\)'' (\cref{def:ocf}), and the convolutions~\(\ou\big({},tt,-\big)\), defined precisely by~\eqref{eq:stsezmuf}, are an exponentially decaying weighted integral of the immediate past history of the forcing~\(w(t)\).

For the simple system~\eqref{eq:stsetoy}, the backward approach developed here gives subtly different results.
\cref{cor:stbackex} asserts that there exist nearby systems of interest, such as that obtained by the time-dependent coordinate transform%
\begin{subequations}\label{eqs:stsetoyxy}%
\begin{align}
x&=X
+ X Y
+\sigma   \ou\big(w,tt,-\big) X
+\tfrac32  X Y^{2}
\,,
\\
y&=Y
+\sigma  \ou\big(w,tt,-\big)
+ \big(X^{2}+2 Y^{2}\big)
+4 \sigma \ou\big(w,tt,-\big) Y
+4 Y^{3}
+\sigma \big(2 \ou\big(w,tt,-\big) X^{2}
\nonumber\\&\qquad{}
-2 \ou\big(\ou\big(w,tt,-\big),tt,-\big) X^{2} -4 \ou\big(w,tt,{}\big) Y^{2}+12 \ou\big(w,tt,-\big) Y^{2}\big)
\,,
\label{eq:stsetoyxyy}
\end{align}
\end{subequations}
together with the associated evolution of the variables  
\begin{subequations}\label{eqs:stsetoyXY}%
\begin{align}
\dot X&=- X^{3}
-\sigma w X
\,,
\label{eq:stsetoyXYX}
\\
\dot Y&=-Y
-4 \sigma w Y
-2 X^{2} Y
\,.
\label{eq:stsetoyXYY}
\end{align}
\end{subequations}
This specific `nearby' system of \eqref{eqs:stsetoyxy}+\eqref{eqs:stsetoyXY} is a truncation of eqns.~(22)--(28) \cite[]{Roberts06k}, derived as an example of forward theorems applied to~\eqref{eq:stsetoy}. 
One may use the computer algebra of a web service \cite[]{Roberts09c} to confirm.
The details of the derivation (\cref{sec:stsenfsmsde}) are not relevant to the core results of this article, other than that the details crucially ensure that \eqref{eqs:stsetoyxy}+\eqref{eqs:stsetoyXY} is in a precise asymptotic sense close to the specified~\eqref{eq:stsetoy}.
Further, in all examples, one may straightforwardly check the `closeness' via first-year undergraduate algebra by simply taking~\(\de t{}\) of the transform~\eqref{eqs:stsetoyxy}, substituting~\eqref{eqs:stsetoyXY}, and simplifying: here one recovers \((\dot x, \dot y)\) of the specified~\eqref{eq:stsetoy} to a difference~\Ord{X^4+Y^4+\sigma^2} as \((X,Y,\sigma)\to0\)\,.
The nearby system~\eqref{eqs:stsetoyxy}+\eqref{eqs:stsetoyXY} is of interest because its  vector field and that of the specified system~\eqref{eq:stsetoy} differ by just~\Ord{X^4+Y^4+\sigma^2}.
That is, we expect that there exists a neighbourhood of the origin in which we may use \eqref{eqs:stsetoyxy}+\eqref{eqs:stsetoyXY} to reasonably predict the dynamics of the specified~\eqref{eq:stsetoy}---predictions such as the following.

\begin{itemize}
\item From the form of~\eqref{eqs:stsetoyXY}, \(Y=0\) is invariant for all~\(X\), and, further, \(Y=0\) is an exact (global) centre manifold for~\eqref{eqs:stsetoyXY}.
Consequently, the coordinate transform~\eqref{eqs:stsetoyxy} with \(Y=0\) gives that an exactly known centre manifold for the system~\eqref{eqs:stsetoyxy}+\eqref{eqs:stsetoyXY} is
\begin{align}
x&= X
+\sigma   \ou\big(w,tt,-\big) X\,,
&
y&=  X^{2}
+\sigma  \ou\big(w,tt,-\big)
+2\sigma   \big(  1
- \ou\big({},tt,-\big) \big)\ou\big(w,tt,-\big) X^{2}.
\label{eq:stsetoyCM}
\end{align}
\cref{cor:stbackex} generally establishes that there are many such systems close to the specified~\eqref{eq:stsetoy}, where the systems possess an exactly known centre manifold, such as the exact~\eqref{eq:stsetoyCM}.

This illustrates a classic distinction between backwards and forwards theory~\cite[e.g.,][]{Grcar2011}: 
forwards theory asserts a centre manifold exists but often can only be known approximately---the order of error terms in~\eqref{eqs:stsetoycm}; 
whereas the backwards approach asserts that there exists an exactly known centre manifold~\eqref{eq:stsetoyCM} of a known system close to that specified.
\emph{Forwards and backwards are complementary views.}

In systems where the centre manifold is of few dimensions, such as this example, one may apply manifold continuation methods to construct numerically exact manifolds over finite domains---supported by forwards theory \cite[e.g.,][]{auto01, England2007, Ziessler2018}.  
Such numerically exact continuation is great for some applications.
But for systems with centre manifolds of more than just a few dimensions, such numerical continuation is not feasible.

\item Regarding the domain of validity:  
extant forward theory would \emph{only} assert that properties of a centre manifold hold in \emph{a neighbourhood} of the origin.  
Many users assume ``a neighbourhood'' {must} be small.
But the neighbourhood of validity could be quite large: \cref{fig:stsesim1} suggests that, for its trigonometric forcing, the domain of validity for this example is at least that part of the \(xy\)-domain plotted. 
\begin{full}
\cref{fig:ctECT} shows an example where the domain of validity extends to infinity in some directions.
\end{full}%
Although \(Y=0\) is an exact \emph{global} centre manifold of the conjugate system~\eqref{eqs:stsetoyXY},
its validity in the \(xy\)-plane, namely~\eqref{eq:stsetoyCM}, is usually limited by the coordinate transform~\eqref{eqs:stsetoyxy}: the coordinate transform must not degenerate.
In contrast to what is possible in extant forward theory, this consideration empowers backwards theory to put a lower bound (\cref{lem:nfftd}) on the size of the \text{domain of validity.}

\item Recall that we usually want trajectories to be exponentially quickly attracted to a centre manifold---that it is emergent. 
We rewrite~\eqref{eq:stsetoyXYY} as \(Y=Y_0\exp\big(-\int_0^t 1+4\sigma w+4X^2\d t\big)\) to demonstrate that \(Y\to0\) exponentially quickly in time for all~\(X(t)\) and for most~\(w(t)\), including \(w=\cos t\) of \eqref{fig:stsesim1}.
This rapid emergence of a centre manifold in the \(XY\)-system then implies that the centre manifold of the \(xy\)-system~\eqref{eqs:stsetoyxy}+\eqref{eqs:stsetoyXY} is rapidly emergent  in the finite domain of validity of the coordinate transform~\eqref{eqs:stsetoyxy} (\cref{pro:stbackem}).
This is a cognate result to that of forward theory which also proves there is an emergent centre manifold, albeit only known in ``a neighbourhood''.
Forwards and backwards give complementary views of this same property.

\end{itemize}

A foundational difference between forward theory and the backwards approach lies in the \cref{def:nfnfim} for invariant manifolds.
In this example, the \emph{algebraic form} of~\eqref{eqs:stsetoyXY} immediately guarantees that \(Y=0\) is a centre manifold for the system.
Then the coordinate transform~\eqref{eqs:stsetoyxy} guarantees that \eqref{eq:stsetoyCM}~is an exact centre manifold for the system~\eqref{eqs:stsetoyxy}+\eqref{eqs:stsetoyXY} in the \(xy\)-plane.
These guarantees apply \emph{for any time interval of interest}, just as long as the system stays within the domain of validity.
In contrast, extant forward theory for non-autonomous systems \emph{requires infinite time integrals to exist} in order to define the invariant manifolds \cite[e.g.,][]{Henry81, Potzsche2006, Haragus2011, Barreira2007, Aulbach2006}.  
Thus for the fundamental definitions to apply, the system has to stay in its neighbourhood of validity for all time. 
Such forward theory must consequently place onerous restrictions on the nonlinearity and time dependence in the systems to which it applies \cite[e.g.,][Hypothesis~3.8(ii)]{Haragus2011}.
Such onerous restrictions are not required in the backwards approach---because of the crucial \text{change in definition.}


\subsection{Discussion}

The fundamental change by the backwards approach in the basic definition of invariant manifolds should empower further developments of backwards theory to apply to a much wider variety of problems than that to which extant forward theory applies.
One example is the first step in creating a backwards theory for \pde{}s by \cite{Hochs2019}.

\begin{full}
Encompassing unstable dynamics with both center and stable is necessary for application to St Venant-like, cylindrical, problems \cite[e.g.,][]{Mielke88b, Mielke91a, Haragus95b}, and to deriving boundary conditions for approximate \pde{}s \cite[e.g.,][]{Roberts92c}.
However, applications of forwards theory in such a general setting is often confounded by two issues.
First, \pde{}s, such as the Navier--Stokes equations for fluid flow, \(\uv_t+\uv\cdot\grad\uv=-\grad p/\rho+\Delta\uv\), typically have unbounded operators, such as~\(\Delta\) and~\(\grad\), and so often do not lie in the scope of most theories of non-autonomous invariant manifold which typically require bounded operators.
The usual extant boundedness requirement \cite[e.g.,][Hypothesis~2.1(i) and 3.8(i)]{Haragus2011} arises from the general necessity of forward and backward convolutions with the semigroup, convolutions that must be continuous in much extant forward theory. 
Non-autonomous theory by \cite{Mielke86} caters for a useful class of unbounded operators, but only proves existence (Thm.~2.1), not emergence nor approximation.
\cite{Henry81} proves existence and emergence (Thms.~6.1.2 and 6.1.4), but only for sectorial operators (Defn.~1.3.1), which limits applicability, and also does not consider unstable modes that are catered for here.
Second, unbounded nonlinearity, such as~\(\uv\cdot\grad\uv\) and as found so often in applications such as fluid instabilities, is not generally covered by the Lipschitz and/or uniformly bounded requirement of most extant forward theory \cite[e.g.,][]{Henry81, Mielke86, Aulbach2000, Chicone97, Haragus2011}.  
So, despite some interesting scenarios having rigorous invariant manifolds beautifully established via strongly continuous semigroup operators and by mollifying nonlinearity \cite[e.g.,][]{Carr81, Vanderbauwhede89}, extant \emph{non-autonomous} forward theory even for finite-D typically imposes preconditions \cite[e.g.,][Hypothesis~3.8(ii)]{Haragus2011} that we relax here via the proposed backward theory: for example, \cref{ass:givensys} only requires nonlinearities to be~\(C^\fp\) for some order~\fp.
\end{full}

One important class of applications is to invariant manifold models of slow modes among fast waves, such as the crucial quasi-geostrophic slow manifold of geophysical fluid dynamics.
\cref{sec:asfss} establishes that although there may be no such slow manifold for a specified system \cite[e.g.,][]{Lorenz87, Vautard86}, nonetheless there generally exist many arbitrarily close systems that do possess true slow manifolds (\cref{pro:sccte}).
The Lorenz86 system \cite[]{Lorenz86} is explored as an illustrative example. 
Generally in dynamics we seek persistent objects, whereas in this type of scenario a slow manifold is not persistent under perturbations \cite[e.g.,][]{Lorenz87, Vautard86}.
Nonetheless, the `slow manifold' of the quasi-geostrophic approximation \cite[e.g.,][]{Leith80},  also of other cognate approximations, have remained hugely valuable practical concepts in science and engineering for several hundred years.
Backwards theory is an innovative new way to provide some theoretical support and understanding of the concept, \text{as \cref{sec:asfss} discusses.}

Many physical systems possess useful symmetries to be maintained for an invariant manifold to be useful in the application.  
Such symmetries may include even\slash odd symmetry in variables, or, importantly in many geophysical\slash engineering applications, the Hamiltonian nature of the given system.
Due to the manyfold possibilities in constructing a suitable coordinate transform~\eqref{eq:stbexsm}, it is usually straightforward to implement the algebra of \cref{sec:stsenfsmsde} to preserve the required symmetries.

\begin{full}
In the example of \cref{sec:stseeg}, the forcing~$w(t)$ could be as regular as a periodic oscillator, as used in \cref{fig:stsesim1}, or could be the output of a deterministic chaotic system~\cite[e.g.,][]{Just01}, or could be random such as the white noise formal derivative of a `{random walk}' {Wiener process}.
The last case of a noisy (stochastic) forcing is both particularly interesting and particularly delicate.
Indeed, the delicacies of the stochastic case determine the methodology we use for all~cases.%
\footnote{Any stochastic interpretation of the algebra is to be the Stratonovich interpretation, not It\^o interpretation, because we use ordinary calculus for all analysis (as did \cite{Arnold98} and \cite{Arnold03}).
}
\end{full}


Extant theory for non-autonomous stochastic systems \cite[e.g.,][]{Arnold03, Mohammed2013, Chekroun2015a} is bedevilled by very rare `escape' events that confound straightforward application of theory to most stochastic systems of interest in applications.
This article is a first step towards overcoming such limitations by beginning to develop this new backward approach to complement extant forward theory. 
\begin{full}
Such backward theory has proved extremely useful in solving linear equations~\cite[e.g.,][]{Grcar2011}, in understanding sensitivity in eigen-problems~\cite[]{Ghosh2018}, and improving our understanding of the validity of perturbation methods \cite[]{Corless2019}.
Indeed \cite{Corless2019} conclude [p.76] that such a backward ``approach is surprisingly useful and clarifies several issues'' including that it ``allows one to directly use approximations taken from divergent series in an optimal fashion without appealing to rules of thumb'', and that it ``interprets the computed solution solution as the exact solution to just as good a model.''
These conclusions also apply here.
\end{full}

Future research is planned extend this approach to further develop new useful invariant manifold theory for partial differential equations \cite[e.g.,][]{Hochs2019}, and/or stochastic dynamical systems, and/or fractional differential equations (\cite{Cong2016} comment that center manifold theory for fractional \textsc{de}s is difficult as they do not generate a semigroup).
The aim of this article is to start establishing results in this direction by laying a foundation of backwards theory in general finite dimensional, non-autonomous, \text{nonlinear systems~\eqref{eq:phys}.}

\section{Construct a conjugacy with a normal form system}
\label{sec:stsenfsmsde}

This section establishes that coordinate transform arguments previously used for fast-slow separations of stochastic dynamics \cite[e.g.,][]{Arnold03, Roberts06k} generalise to also encompasses quite general non-autonomous dynamics.  
This section establishes firstly that time dependent coordinate transforms exist from systems with decoupled center modes from hyperbolic modes (making the time dependent invariant manifolds easy to extract).
Such a generalised Hartman--Grobman existence theory has previously been established \cite[\S4, e.g.]{Aulbach2000}: the qualitative difference here is that the framing is distinct because here we provide results to prove the new backward theory of \cref{sec:nfeep}.
Secondly, this section establishes that although anticipation of the time dependence may be necessary in the full transform (as in the~\(\ou\big(w,tt,{}\big)\) of the example~\eqref{eq:stsetoyxyy}), no anticipation need appear on the center manifold itself.

This section straightforwardly utilises and synthesises previous research.
The novelty here is combining and extending the many aspects together: 
covering cases of dynamics with stable, center and unstable dynamics; 
covering cases where the center dynamics may have non-zero growth\slash decay rate; 
covering non-autonomous effects in a manner that best suit long-time modelling (and potentially stochastic effects).%
\footnote{\cite{Haragus2011} restrict their normal form Theorem~5.5.2 to periodic time dependence.}
Although the construction technique may be useful in some scenarios,
the real importance of this section is that the results underpin the new backward theory of \cref{sec:nfeep} that leads to \cref{cor:stbackex,pro:stbackem}.

\begin{assumption}\label{ass:givensys}
The given separated system~\cref{eqs:stsesde} is to satisfy the following:
\begin{enumerate}
 \item the matrix~$A$ has eigenvalues~\(\alpha_1,\ldots,\alpha_m\), possibly complex, with `small' real-part bounded by \(|\Re\alpha_i|\leq\alpha\);%
\footnote{In principle, the matrices~$A$, $B$ and~\(C\) could also depend upon time~\cite[e.g.,][]{Chicone97, Aulbach2000, Potzsche2006}.  
When the Lyapunov exponents of the corresponding linear dynamics are near-zero, negative and positive respectively,
then the invariant manifolds should still exist and have nice properties.
However, we focus on the algebraically tractable case when the basic linear operators~$A$, $B$ and~\(C\) are constant.
Indeed, in constructing nontrivial non-autonomous slow manifolds, the \emph{only} definite example that I recall that has not been based upon constant linear operators is just one example by \cite{Potzsche08b}.
}

    \item the matrix~$B$ has eigenvalues $\beta_1,\ldots,\beta_n$, possibly complex, with `large' negative real-part bounded above by $\Re\beta_j\leq-\beta<0$\,;

    \item similarly, the matrix~$C$ has eigenvalues $\gamma_1,\ldots,\gamma_\ell$, possibly complex, with `large' positive real-part bounded below by $0<\beta\leq\Re\gamma_k$ (for simplicity the bound~\(\beta\) is common to both \(B\) and~\(C\));%
    
    \item there are spectral gaps as these rate bounds are to satisfy \(\beta>(2\fp-1)\alpha\) for some integer order \(\fp\geq2\);

    \item  for the same order~\fp, functions~$\fv$, $\gv $ and~\hv\ are~$\cC^\fp(d)$ in some finite, connected, domain~$d$ containing the origin, and the functions are `strictly nonlinear' in that they and their first derivatives in~\(\xv,\yv,\zv\) are all zero at the origin (that is, \Orc{2} according to \cref{def:ocf}).
    
\begin{full}
 \item the time dependence may be written, implicitly or explicitly, as a linear combination of some number of independent forcing processes~$w_l(t)$ (which in future extension to a stochastic case would be {Stratonovich} `white noises', or even Marcus Levy flights \cite[e.g.,][]{Chechkin2014}). 
\end{full}
    
\end{enumerate}
\end{assumption}

\paragraph{A time dependent coordinate transform}

Analogous to the introductory example of \cref{sec:stseeg}, we relate the generic \ode\ system~\cref{eqs:stsesde} in~$\uv=(\xv,\yv,\zv)$ to a conjugate system in new coordinates $\Uv=(\Xv,\Yv,\Zv)$ by a time dependent, near identity, coordinate transform
\begin{equation}
    \xv=\xv(t,\Xv,\Yv,\Zv),\quad
    \yv=\yv(t,\Xv,\Yv,\Zv),\quad
    \zv=\zv(t,\Xv,\Yv,\Zv).
    \label{eq:stsexform}
\end{equation}
Such a time dependent coordinate transform is to be chosen such that the system~\cref{eqs:stsesde} is closely approximated by the `simpler' conjugate system.
\begin{full}
The variables \((\Xv,\Yv,\Zv)\) are sometimes called intrinsic coordinates \cite[(26), e.g.]{Brunton2016}.
\end{full}
We specifically seek to construct coordinate transforms that simplify the \ode{}s in the sense of nonlinearly separating the dynamics of the centre, stable, and unstable variables.
In this section the meaning of `closely approximates' is in a precisely defined asymptotic sense.  

\begin{definition}\label{def:ocf}
\begin{enumerate}
\item We use the following precise meanings for asymptotic terminology \cite[e.g.,][pp.78,318 and \S2.3.1, resp.]{Bender81, Roberts2014a}.
For any \(f,g,h,x,x_*\), and \(x_*\)~may be~\(\pm\infty\)\,, we write ``\(f(x)=g(x)+\Ord{h(x)}\) as \(x\to x_*\)'' to mean that \([f(x)-g(x)]/h(x)\) is bounded as \(x\to x_*\)\,.

We define that \emph{asymptotically close to order~\(p\)} means, for \(f,g,x,x_*\) apparent from the context, that \(f(x)=g(x)+\Ord{|x|^p}\) as \(x\to x_*\)\,.   
Then \emph{asymptotically close} means asymptotically close to order~\(p\) for some appropriate order~\(p>0\).
And \emph{exponentially close} means asymptotically close to order~\(p\) for every order~\(p>0\).

\item For every~\(p\), let the order symbol \Orc{p} denote  ``\Ord{|\Xv|^p+|\Yv|^p+|\Zv|^p} as $(\Xv,\Yv,\Zv)\to\ov $''.  
For example, a polynomial term $X^mY^nZ^\ell=\Orc{p}$ if and only if $m+n+\ell\geq p$\,.

\item Choose a threshold rate~\muc\ in the spectral gap, \(\fp\alpha\leq\muc<\beta-(\fp-1)\alpha\) (a gap which exists by \cref{ass:givensys}).  
Then for terms\slash modes\slash variables\slash rates associated with the exponential~\(e^{\lambda t}\), classify them as 
\emph{center} if \(|\Re\lambda|\leq\muc\)\,, 
\emph{stable} if \(\Re\lambda<-\muc\)\,,  
\emph{unstable} if \(\Re\lambda>\muc\)\,, 
\emph{hyperbolic} if \(|\Re\lambda|>\muc\)\,, 
\emph{fast} if \(|\lambda|>\muc\)\,, and
\emph{slow} if \(|\lambda|\leq\muc\)\,. 
\end{enumerate}
\end{definition}

When all the center modes are precisely neutral (the \(A\)-eigenvalue bound \(\alpha=0\)), then the choice \(\muc=0\) recovers the most commonly used classification of center, stable, unstable, slow and fast variables \cite[e.g.,][]{Carr81, Mielke86, Vanderbauwhede89, Haragus2011}.

The constructive series argument of \cite{Roberts06k} [Proposition~1] for stochastic non-autonomous systems is generalised in this section both to include unstable variables and also to cater for threshold rate \(\muc\neq0\)\,. 
\cref{sec:sttsd,sec:sttfd,sec:sttud} contribute to the following \cref{thm:stsenfp} which establishes the existence of the `nearby' dynamical systems invoked in \cref{cor:stbackex,pro:stbackem}.

\begin{theorem} \label{thm:stsenfp}
For every given order~\(p\), \(2\leq p\leq\fp\)\,, there exists a near identity, polynomial, time dependent, coordinate transformation~\cref{eq:stsexform} and a corresponding polynomial normal form system
\begin{subequations}\label{eqs:stsesdenf}%
\begin{align}
    \dot{\Xv}&=A\Xv+\Fv (t,\Xv,\Yv,\Zv)
    \nonumber\\
    &=A\Xv+\Fv_c(t,\Xv)+\cF(t,\Xv,\Yv,\Zv)\Yv\Zv,
     \label{eq:stsesdenfxx}\\
    \dot{\Yv}&=\big[ B+G(t,\Xv,\Yv,\Zv) \big]\Yv,
    \label{eq:stsesdenfyy}\\
    \dot{\Zv}&=\big[ C+H(t,\Xv,\Yv,\Zv) \big]\Zv,
    \label{eq:stsesdenfzz}
\end{align}
\end{subequations}
which together is approximately conjugate to the non-autonomous system~\cref{eqs:stsesde},
where~\Fv, \(G\Yv\) and~\(H\Zv\) are~\Orc{2}  (and where \cF~is a rank three tensor),
and where, by construction, the difference between the system~\cref{eq:stsexform}+\cref{eqs:stsesdenf} and the system~\cref{eqs:stsesde} is~\Orc{p}.

Significantly, $\Fv $, $G$ and~\(H\) need only contain fast time `memory\slash anticipation' integrals in terms that are quadratic, or higher power, in the non-autonomous terms: no such integrals are \emph{needed} in the linear terms.
\end{theorem}


\begin{proof} 
We prove \cref{thm:stsenfp} via algebraic results deduced in the three subsequent subsections.

As a preliminary step, linearly change basis for each of \xv~and~\Xv, \yv~and~\Yv, and \zv~and~\Zv, so that the systems~\cref{eqs:stsesde} and~\cref{eqs:stsesdenf} then have matrices~\(A\), \(B\) and~\(C\) in upper-triangular form (a Jordan form or a Schur decomposition are two examples).
Then prove by induction.
First, the lemma is trivially true for order $p=2$ under the identity transform $\xv=\Xv$, $\yv=\Yv$, and $\zv=\Zv$ with normal form system $\dot \Xv=A \Xv$, $\dot \Yv=B \Yv$, and $\dot \Zv=C \Zv$.
Higher order~\(p\) only changes this identity transform by higher order, polynomial, corrections and so in this sense they are all `near identity'.  

Second, assume there exists such a coordinate transform and normal form for some~$p$, $2\leq p<\fp$\,; that is, the residuals of~\cref{eqs:stsesde} are~\Orc{p} upon substituting \cref{eq:stsexform}+\cref{eqs:stsesdenf}.
Seek corrections~\Orc{p}, indicated by hats, that lead to residuals of~\Orc{p+1}:  let 
\begin{subequations}\label{eqs:corsub}%
\begin{align}
&\xv=\xv(t,\Xv,\Yv,\Zv)+\hat \xv(t,\Xv,\Yv,\Zv),
\\& \yv=\yv(t,\Xv,\Yv,\Zv)+\hat \yv(t,\Xv,\Yv,\Zv),
\\& \zv=\zv(t,\Xv,\Yv,\Zv)+\hat \zv(t,\Xv,\Yv,\Zv),
\\\text{where}\quad& \dot \Xv=A \Xv+\Fv(t,\Xv,\Yv,\Zv)+\hat\Fv(t,\Xv,\Yv,\Zv),
\\& \dot \Yv=B \Yv+G(t,\Xv,\Yv,\Zv)\Yv+\hat G(t,\Xv,\Yv,\Zv)\Yv,
\\& \dot \Zv=C \Zv+H(t,\Xv,\Yv,\Zv)\Zv+\hat H(t,\Xv,\Yv,\Zv)\Zv. 
\end{align}
\end{subequations}
\cref{sec:sttsd} establishes suitable~\(\hat\xv\) and~\(\hat\Fv\) exist for the center components;
\cref{sec:sttfd} establishes  suitable~\(\hat\yv\) and~\(\hat G\) exist for the fast stable components;
and \cref{sec:sttud} establishes  suitable~\(\hat\zv\) and~\(\hat H\) exist for the fast unstable components.
Hence choosing such a suitable~\cref{eq:stsexform}+\cref{eqs:stsesdenf} satisfies the system~\cref{eqs:stsesde} to residuals~\Orc{p+1}.
By Proposition~3.6 of \cite{Potzsche2006}, it follows that the difference between a system~\cref{eq:stsexform}+\cref{eqs:stsesdenf} and the system~\cref{eqs:stsesde} is also~\Orc{p+1}.
By induction, \cref{thm:stsenfp} holds for systems~\cref{eqs:stsesde} with upper-triangular matrices~\(A\), \(B\) and~\(C\).
By reverting the preliminary change of basis, \cref{thm:stsenfp} holds for general matrices.
\end{proof}

\paragraph{Require convolutions} 
In general we need to invoke convolutions in time of the direct \(t\)-dependence in terms (for `frozen'~\Xv, \Yv\ and~\Zv), convolutions that depend upon the eigenvalues of the three matrices (as in the example~\eqref{eqs:stsetoyxy}).  
These convolutions encapsulate either memory or anticipation of the time-dependence over the various short time scales in the hyperbolic variables~$\yv$ and~\zv. 
For any parameter~$\mu$ (possibly complex), for sufficiently well 
behaved (e.g., integrable) time dependent functions~$a(t)$, and for two fixed chosen times~\(t_\pm\), define the \emph{convolution}%
\begin{equation}
    \Z{\mu}a:=
    \begin{cases}
        \int_{t_-}^t \exp[\mu(t-\tau)]a(\tau)\,d\tau\,,
        &\Re\mu<-\muc\,, \\[1ex]
        \int_t^{t_+} \exp[\mu(t-\tau)]a(\tau)\,d\tau\,,
        &\Re\mu>+\muc\,, \\[1ex]
        \text{undefined,}&|\Re\mu|\leq\muc\,,             
    \end{cases}
    \label{eq:stsezmuf}
\end{equation}
where \cref{def:ocf} determines the chosen cut-off rate~\muc\ in this definition.
Such convolutions are bounded for at least bounded~$a(t)$: more specifically, we \emph{choose} to invoke such convolutions only when the result is bounded (so that the asymptotic properties are uniform over times of interest).  
Our derivations only require that such convolutions~\cref{eq:stsezmuf} exist and are differentiable.
We use that the convolutions give a bounded solution to the~\ode
\begin{align}&
    \frac{d\ }{dt}(\Z{\mu}a)-\mu(\Z{\mu}a)=-(\sgn\Re\mu)a:
    \label{eq:stseddtconv}
\end{align}
with $\Re\mu<0$ the convolution~$\Z\mu$ integrates over the past (a memory); 
with $\Re\mu>0$ the convolution~$\Z\mu$ integrates into the future (anticipation); 
both integrate over a short hyperbolic-time scale of order~$1/|\Re\mu|$.

The freedom to choose integration end-points~\(t_\pm\) is just one of the contributions to non-uniqueness in the constructed coordinate transform.  
To ensure boundedness for every appropriate rate~\(\mu\), such as in scenarios where matrices have widely varying eigenvalues, then we need \(t_-\leq t\leq t_+\) and hence we might choose~\(t_\pm\) to be the end-points of the time interval of interest, \(\TT_\mu=[t_-,t_+]\).
Alternatively, in scenarios where the time variations in the system, represented by~\(a(t)\) in~\eqref{eq:stsezmuf}, are known to be well enough behaved for all time, then we might choose \(t_\pm=\pm\infty\).
The differences between such choices only affects the constructed coordinate transform, and hence the derived invariant manifolds, in `boundary layers' in time near the start and end of~\(\TT_\mu\)---boundary layers that are of thickness in time of order~$1/|\Re\mu|$.

\subsection{Transform the center dynamics}
\label{sec:sttsd}

This subsection contributes to the proof of \cref{thm:stsenfp} via straightforward extensions of established arguments.
For the center dynamics, each iteration towards constructing a time dependent coordinate transform substitutes sought corrections~\cref{eqs:corsub} (for the transform and the evolution) into the governing \ode~\cref{eq:stsesdex} for the center variables.

\paragraph{A homological equation governs corrections}
First, in the right-hand side of the \ode~\cref{eq:stsesdex}, the nonlinear function \(\fv(t,\xv+\hat\xv,\yv+\hat\yv,\zv+\hat\zv)=\fv(t,\xv,\yv,\zv)+\Orc{p+1}\) by a multivariate Taylor's theorem since \(\hat\xv,\hat\yv,\hat\zv=\Orc{p}\) and derivatives \(\fv_\xv,\fv_\yv,\fv_\zv=\Orc{1}\) by \cref{ass:givensys}.
Here and throughout, such subscripts represent partial derivatives: for example, the Jacobian matrix \(\fv_\xv=\begin{bmatrix} \D{x_j}{f_i} \end{bmatrix}\).

Second, the time derivative on the left-hand side of~\cref{eq:stsesdex} is more complicated: by the chain rule the corrected time derivative
\begin{align*}
\dot \xv&=(\xv_t+\hat\xv_t)+(\xv_\Xv+\hat \xv_\Xv)\dot \Xv+(\xv_\Yv+\hat \xv_\Yv)\dot \Yv +(\xv_\Zv+\hat \xv_\Zv)\dot \Zv
\\&=(\xv_t+\hat\xv_t)+(\xv_\Xv+\hat \xv_\Xv)(A \Xv+\Fv+\hat \Fv)+(\xv_\Yv+\hat \xv_\Yv)(B+G+\hat G)\Yv 
\\&\quad{}+(\xv_\Zv+\hat \xv_\Zv)(C+H+\hat H)\Zv
\\&=(\xv_t+\hat\xv_t)+{\xv_\Xv(A \Xv+\Fv)+\xv_\Yv(B+G)\Yv+\xv_\Zv(C+H)\Zv}
\\&\quad{}
+\underbrace{\hat \xv_\Xv(A \Xv+\Fv)}_{=A \Xv\hat \xv_\Xv+\Orc{p+1}}
+\underbrace{\xv_\Xv\hat \Fv}_{=\hat \Fv+\Orc{p+1}}
+\underbrace{\hat \xv_\Xv\hat \Fv}_{\Orc{2p-1}}
+\underbrace{\hat \xv_\Yv(B+G)\Yv}_{=B \Yv\hat \xv_\Yv+\Orc{p+1}}
\\&\quad{}
+\underbrace{\xv_\Yv\hat G\Yv}_{\Orc{p+2}}
+\underbrace{\hat \xv_\Yv\hat G\Yv}_{\Orc{2p}} 
+\underbrace{\hat \xv_\Zv(C+H)\Zv}_{=C \Zv\hat \xv_\Zv+\Orc{p+1}}
+\underbrace{\xv_\Zv\hat H\Zv}_{\Orc{p+2}}
+\underbrace{\hat \xv_\Zv\hat H\Zv}_{\Orc{2p}} \,.
\end{align*}
As indicated, omit products of small corrections, and approximate the coefficients of the remaining small corrections by their leading order term (for example, $\D{\Xv}{\xv}\approx I$ and $\D{\Yv}{\xv}\approx \D\Zv\xv\approx0$).
Third, equating the two sides with a little rearrangement, the $\xv$-equation~\cref{eq:stsesdex} becomes
\begin{align}&
\hat\xv_t-A\hat \xv+
\hat \xv_\Xv A\Xv +\hat \xv_\Yv B\Yv+\hat\xv_\Zv C\Zv+\hat \Fv
\nonumber\\&{}
=\underbrace{A \xv+\fv\underbrace{{}-\xv_t-\xv_\Xv(A \Xv+\Fv)-\xv_\Yv(B+G)\Yv
-\xv_\Zv(C+H)\Zv}_{-d\xv/dt}}_{\res_{\protect\cref{eq:stsesdex},p}}
+\Orc{p+1}.
\label{eq:res3xp}
\end{align}

\paragraph{Solve homological equation~\cref{eq:res3xp} to find corrections}
As explained at the end of this subsection, at the expense of generally increasing the number of required iterations, we neglect in the left-hand side the off-diagonal terms in the upper triangular matrices~\(A\), \(B\) and~\(C\).
Also, by \cref{ass:givensys}, the induction assumption, and the multivariate Lagrange Remainder Theorem in~\((\Xv,\Yv,\Zv)\), the residual on the right-hand side of~\cref{eq:res3xp} is written as a polynomial 
\(\res_{\protect\cref{eq:stsesdex},p}
=\sum(\text{terms})+\Orc{p+1}\) in which each term is of the form
of a \(p\)th~order polynomial term with coefficient vector~\(\av(t)\), namely
\( \av(t) \Xv^\pv\Yv^\qv\Zv^\rv
=\av(t)\prod_{i=1}^m X_i^{p_i}\,
\prod_{j=1}^n Y_j^{q_j}\,
\prod_{k=1}^\ell Z_k^{r_k}
\)
for exponent multi-indices \(\pv=(p_1,\ldots,p_m)\in \mathbb N_0^m\), and similarly for~\qv\ and~\rv, such that \(|\pv|+|\qv|+|\rv|=p\).

For each such \(p\)th~order term on the right-hand side, the $\kappa$th~component of the homological equation~\cref{eq:res3xp} is then 
\begin{align}&
    \hat F_\kappa+\D t{\hat x_\kappa}-\alpha_\kappa\hat x_\kappa 
    +\sum_{i=1}^m \alpha_iX_i\D{X_i}{\hat x_\kappa} 
    +\sum_{j=1}^n \beta_jY_j\D{Y_j}{\hat x_\kappa} 
    +\sum_{k=1}^\ell \gamma_kZ_k\D{Z_k}{\hat x_\kappa} 
    \nonumber\\&
    =a_\kappa(t) \Xv^\pv\Yv^\qv\Zv^\rv .
    \label{eq:stsexd}
\end{align}
Because of the special form of the `homological' operator on the left-hand side of~\cref{eq:stsexd}, for each right-hand side term we seek corresponding corrections $\hat F_\kappa=\ff(t)\Xv^{\pv}\Yv^{\qv}\Zv^\rv$ and $ \hat x_\kappa=\fx(t)\Xv^{\pv}\Yv^{\qv}\Zv^\rv$.
Then~\cref{eq:stsexd} becomes
\begin{equation}
    \ff+\dot \fx -\mu \fx=a_\kappa(t)
    \quad\text{where }
    \mu:= \alpha_\kappa-\sum_{i=1}^m p_i\alpha_j
    -\sum_{j=1}^n q_j\beta_j
    -\sum_{k=1}^\ell r_k\gamma_k\,.
    \label{eq:stsemux}
\end{equation}
There are many possible ways to choose the coordinate transform corrections~$\fx$ and~$\ff$ as equation~\cref{eq:stsemux} forms an underdetermined system.
It is up to our qualitative aims to decide what corrections are desirable to implement among all the possibilities.  
One may make choices to preserve certain symmetries or topological properties.
The following choices lead to, in some sense, the minimal coordinate transform necessary.

Two cases typically arise depending upon the real part of the rate~$\mu$.
\begin{enumerate}
 \item The cases when $\mu$~is hyperbolic, \(|\Re\mu|>\muc\), occur when at least one of the exponents in~\qv\ or~\rv\ is non-zero, see the next paragraph.
Accepting possible anticipation in the coordinate transform (such as the second to last term in~\eqref{eq:stsetoyxyy}), we assign $\fx=\Z\mu \av$\,, and do not change the $\Xv$~evolution, $\ff=0$\,.
The convolution in the coordinate correction~\fx\ only involves convolutions over hyperbolic-time scales.

Conversely, if \(\qv=\rv=\ov\), then from~\cref{eq:stsemux}
\begin{eqnarray*}
|\Re\mu|&=&\left|
\Re\alpha_\kappa-\sum_{i=1}^m p_i\Re\alpha_i
    +\sum_{j=1}^n 0\Re(-\beta_j)
    -\sum_{k=1}^\ell 0\Re\gamma_k \right|
    \\&\leq&(|\pv|+1)\alpha
    =(p+1)\alpha
    \leq\fp\alpha
    \leq\muc\,.
\end{eqnarray*}
Thus \(\mu\) is not hyperbolic by the separation \cref{def:ocf}.
Hence for hyperbolic~\(\mu\), it must be that at least one exponent in~\qv\ or~\rv\ is non-zero.

\item \label{i:stsepx} The `resonant' center case,  $|\Re\mu|\leq\muc$\,, only arises in two circumstances. 
\begin{itemize}
\item Firstly, when the hyperbolic exponents \(\qv=\rv=\ov\) as justified above, in which case there are no hyperbolic variables~\Yv\ and~\Zv\ in the term.
\item Secondly, it may arise when both \(\qv\neq\ov\) and \(\rv\neq\ov\), in which case the term always has at least one stable variable~\(Y_j\) and at least one unstable variable~\(Z_k\).
To eliminate the other possibilities consider the possibilty \(\qv\neq\ov=\rv\), then from~\cref{eq:stsemux}, 
\begin{eqnarray*}
\Re\mu&=&\Re\alpha_\kappa-\sum_{i=1}^m p_i\Re\alpha_i
    -\sum_{j=1}^n q_j\Re\beta_j
    -\sum_{k=1}^\ell 0\,\Re\gamma_k 
\\&\geq&-(|\pv|+1)\alpha+|\qv|\beta 
\\&\geq& -p\alpha+\beta \geq-(\fp-1)\alpha+\beta >\muc\,.
\end{eqnarray*}
Alternatively, when \(\rv\neq\ov=\qv\) then a similar derivation gives the bound \(\Re\mu\leq-\beta+(\fp-1)\alpha<-\muc\).
Thus these two cases cannot give center rates~\(\mu\).
\end{itemize}
This center case of small~\(|\Re\mu|\) implies that convolutions \(\Z\mu a_\kappa(t)\) are generally large due to the relatively large support of the exponential~\(e^{\mu t}\) in the convolution.
We need to avoid the possibility of such large terms.
Thus at first sight in solving~\cref{eq:stsemux}, $\ff+\dot\fx-\mu\fx=a_\kappa$\,, a generic acceptable solution is to correct the \Xv~evolution with $\ff=a_\kappa$ and leave the coordinate transform unchanged with $\fx=0$\,.
    
But recall that we want to avoid hyperbolic-time integrals in the center evolution~\(\dot\Xv\); that is, we want to avoid assigning to~\ff\ terms in \(\Z\nu{\tilde a}(t)\).
Consider the case when the forcing~$a_\kappa(t)$ has the form of a hyperbolic-time convolution $a_\kappa=\Z\nu{\tilde a}(t)$ for some~$\tilde a(t)$ and some rate~\(\nu\).
From~\cref{eq:stseddtconv} deduce 
\begin{eqnarray*}
&&\dot a_\kappa=\nu a_\kappa-(\sgn\Re\nu)\tilde a
\\&\iff&e^{\mu t}\de t{}(e^{-\mu t}a_\kappa)=\dot a_\kappa-\mu a_\kappa=(\nu-\mu) a_\kappa-(\sgn\Re\nu)\tilde a
\\&\iff&
a_\kappa=\frac{\sgn\Re\nu}{\nu-\mu}\tilde a
+\frac{1}{\nu-\mu}e^{\mu t}\de t{}(e^{-\mu t}a_\kappa)
\end{eqnarray*}
Since~\cref{eq:stsemux} may be written as \(\ff+e^{\mu t}\de t{}(e^{-\mu t}\fx)=a_\kappa\)\,,
to avoid hyperbolic-time memory integrals in the center $\Xv$~evolution, set $\ff={\tilde a}(t)(\sgn\Re\nu)/(\nu-\mu)$ and $\fx=a_\kappa/(\nu-\mu)=(\Z\nu{\tilde a})/(\nu-\mu)$ (which assigns the hyperbolic-time convolution to the coordinate transform).
If ${\tilde a}(t)$~in turn is a hyperbolic-time convolution, then continue the above splitting recursively.
    
When the coefficient~$a_\kappa(t)$ is a quadratic product of convolutions, then one may choose similar splittings to eliminate all hyperbolic integrals from the center variables \emph{except} for terms with coefficients of the form ${\tilde a}_1(t)\Z\nu{\tilde a}_2(t)$ where ${\tilde a}_1$~has no convolutions.
Algebraic transformations cannot eliminate such terms \cite[]{Chao95}. 
For now accept such quadratic non-autonomous terms.%
\footnote{Similar considerations apply to higher order terms in the time dependence, but for simplicity we stop at quadratic effects.}
Such quadratic terms encode mechanisms that cause rapid time fluctuations to generate potentially important mean drift effects on \text{the macroscale dynamics.}

\end{enumerate}

This completes the center variables contribution to the inductive proof of \cref{thm:stsenfp}.

\paragraph{Off-diagonal neglect}
Recall the earlier recommendation to omit a term in~\cref{eq:res3xp}: the term $(\D{\Xv}{\hat{\xv}})A\Xv$, equivalently $(\D{X_i}{ \hat x_\kappa})A_{i,j}X_j$, should appear in the left-hand side.
However, its omission is acceptable when the matrix~$A$ is upper triangular (invoked at the start of the proof of \cref{thm:stsenfp}) as then any term introduced which involves~$X_i$ only generates extra terms which are lower order in~$X_i$.
Such extra terms increase the order of~$X_j$ for $j>i$\,, through the off-diagonal terms in~$A$.
However, successive iterations generate new terms involving only fewer factors of~$X_i$ and so iteration steadily accounts for the introduced terms.
Similarly for the $\Yv$ and~\Zv\ variables when the linear operators~$B$ and~\(C\) are triangular.
Discussing equation~\cref{eq:stseyd} for corrections is sufficient.

\subsection{Transform the rapid stable dynamics}
\label{sec:sttfd}

This subsection contributes a second part to the proof of \cref{thm:stsenfp}.
For the stable dynamics, each iteration towards constructing a time dependent coordinate transform substitutes sought corrections~\cref{eqs:corsub} to the transform and the evolution into the governing \ode~\cref{eq:stsesdey} for the stable variables.

\paragraph{A homological equation guides corrections}
First, in the right-hand side of the \ode~\cref{eq:stsesdey}, the nonlinear function \(\gv(t,\xv+\hat\xv,\yv+\hat\yv,\zv+\hat\zv)=\gv(t,\xv,\yv,\zv)+\Orc{p+1}\) by a multivariate Taylor's theorem since \(\hat\xv,\hat\yv,\hat\zv=\Orc{p}\) and derivatives \(\gv_\xv,\gv_\yv,\gv_\zv=\Orc{1}\) by \cref{ass:givensys}.

Second, the time derivative on the left-hand side of~\cref{eq:stsesdey} is more complicated: by the chain rule the corrected time derivative
\begin{align*}
\dot \yv&=(\yv_t+\hat\yv_t)+(\yv_\Xv+\hat \yv_\Xv)\dot \Xv+(\yv_\Yv+\hat \yv_\Yv)\dot \Yv +(\yv_\Zv+\hat \yv_\Zv)\dot \Zv
\\&=(\yv_t+\hat\yv_t)+(\yv_\Xv+\hat \yv_\Xv)(A \Xv+\Fv+\hat \Fv)+(\yv_\Yv+\hat \yv_\Yv)(B+G+\hat G)\Yv 
\\&\quad{}+(\yv_\Zv+\hat \yv_\Zv)(C+H+\hat H)\Zv
\\&=(\yv_t+\hat\yv_t)+{\yv_\Xv(A \Xv+\Fv)+\yv_\Yv(B+G)\Yv+\yv_\Zv(C+H)\Zv}
\\&\quad{}
+\underbrace{\hat \yv_\Xv(A \Xv+\Fv)}_{=A \Xv\hat \yv_\Xv+\Orc{p+1}}
+\underbrace{\yv_\Xv\hat \Fv}_{=\Orc{p+2}}
+\underbrace{\hat \yv_\Xv\hat \Fv}_{\Orc{2p}}
+\underbrace{\hat \yv_\Yv(B+G)\Yv}_{=B \Yv\hat \yv_\Yv+\Orc{p+1}}
\\&\quad{}
+\underbrace{\yv_\Yv\hat G\Yv}_{\hat G\Yv+\Orc{p+1}}
+\underbrace{\hat \yv_\Yv\hat G\Yv}_{\Orc{2p-1}} 
+\underbrace{\hat \yv_\Zv(C+H)\Zv}_{=C \Zv\hat \yv_\Zv+\Orc{p+1}}
+\underbrace{\yv_\Zv\hat H\Zv}_{\Orc{p+2}}
+\underbrace{\hat \yv_\Zv\hat H\Zv}_{\Orc{2p}} \,.
\end{align*}
As indicated, omit products of small corrections, and approximate coefficients of the remaining small corrections by their leading order term (for example, $\D{\Yv}{\yv}\approx I$ and $\D{\Xv}{\yv}\approx\D\Zv\yv\approx 0$).
Third, equating the two sides with a little rearrangement, the $\yv$-equation~\cref{eq:stsesdey} becomes
\begin{align}&
\hat\yv_t-B\hat \yv+
\hat \yv_\Xv A\Xv +\hat \yv_\Yv B\Yv+\hat\yv_\Zv C\Zv+\hat G\Yv
\nonumber\\&{}
=\underbrace{B \yv+\gv\underbrace{{}-\yv_t-\yv_\Xv(A \Xv+\Fv)-\yv_\Yv(B+G)\Yv
-\yv_\Zv(C+H)\Zv}_{-d\yv/dt}}_{\res_{\protect\cref{eq:stsesdey},p}}
+\Orc{p+1}.
\label{eq:res3yp}
\end{align}

\paragraph{Solve homological equation~\cref{eq:res3yp} to find corrections}
Recall we neglect in the left-hand side the off-diagonal terms in the triangular matrices~\(A\), \(B\) and~\(C\).
By \cref{ass:givensys} and the induction assumption, the residual on the right-hand side of~\cref{eq:res3yp} is a polynomial 
\(\res_{\protect\cref{eq:stsesdey},p}
=\sum(\text{terms})+\Orc{p+1}\) in which each \(p\)th~order term is of the form
a polynomial term with coefficient vector~\(\bv(t)\), namely
\(\bv(t) \Xv^\pv\Yv^\qv\Zv^\rv\) such that \(|\pv|+|\qv|+|\rv|=p\).

For each such \(p\)th~order term on the right-hand side, the $\kappa$th~component of the homological equation~\cref{eq:res3yp} is, for \(\hat G_\kappa\) the \(\kappa\)th component of~\(\hat G\Yv\),  
\begin{align}&
    \hat G_\kappa+\D t{\hat y_\kappa}-\beta_\kappa\hat y_\kappa 
    +\sum_{i=1}^m \alpha_iX_i\D{X_i}{\hat y_\kappa} 
    +\sum_{j=1}^n \beta_jY_j\D{Y_j}{\hat y_\kappa} 
    +\sum_{k=1}^\ell \gamma_kZ_k\D{Z_k}{\hat y_\kappa} 
    \nonumber\\&
    =b_\kappa(t) \Xv^\pv\Yv^\qv\Zv^\rv .
    \label{eq:stseyd}
\end{align}
Because of the special form of the `homological' operator on the left-hand side of~\cref{eq:stseyd}, for each right-hand side term seek corresponding corrections $\hat G_\kappa=\fg(t)\Xv^{\pv}\Yv^{\qv}\Zv^\rv$ and $ \hat y_\kappa=\fy(t)\Xv^{\pv}\Yv^{\qv}\Zv^\rv$.
Then~\cref{eq:stseyd} becomes
\begin{equation}
    \fg+\dot \fy -\mu \fy=b_\kappa(t)
    \quad\text{where }
    \mu:= \beta_\kappa-\sum_{i=1}^m p_i\alpha_i
    -\sum_{j=1}^n q_j\beta_j
    -\sum_{k=1}^\ell r_k\gamma_k\,.
    \label{eq:stsemuy}
\end{equation}
Among the many possible ways to choose the coordinate transform corrections~$\fy$ and~$\fg$, the following choices lead to a suitable coordinate transform necessary to achieve our modelling aims.

Three cases arise depending upon the real part of the rate~$\mu$.
\begin{enumerate}
 \item Consider the resonant case of center~$\mu$, \(|\Re\mu|<\muc\)\,. 
To satisfy~\cref{eq:stsemuy}, namely $\fg +\dot \fy-\mu\fy=b_\kappa$, the mean and some types of fluctuations in~$b_\kappa(t)$ must be generally assigned to~$\fg$ as generally they would give rise to large secular terms in~$\fy$. 
For example, when the fluctuating part of~$b_\kappa(t)$ is noisy (stochastic) then integrating it into the coordinate transform~$\fy$ would almost surely generate unallowable square-root growth. 
Thus the generic solution is $\fg =b_\kappa$ and $\fy=0$\,, that is, assign $b_\kappa(t)\Xv^{\pv } \Yv^{\qv }\Zv^\rv$ to the $\Yv$~evolution and nothing into the coordinate transform~$\yv$.
 
Since~$\Re\beta_\kappa\leq-\beta$\,,
this case of center~$\mu$ only arises when at least one of the exponents~$\qv $ of~$\Yv$ is positive in order for the sum in~\cref{eq:stsemuy} to have a center real-part.
Hence, there will be at least one $Y_j$~factor in updates~$\hat{G}$ to the $\Yv$-evolution, and so we maintain the form \(G\Yv\) in the right-hand side of~\(\dot\Yv\).

 \item For stable $\Re\mu<-\muc$\,, a solution of~\cref{eq:stsemuy} is to place all the forcing into the coordinate transform, $\fy=\Z\mu b_\kappa$\,, and not to introduce a component into the $\Yv$-evolution, $\fg =0$\,.
As $\Re\mu<-\muc$\,, the convolution is over the past history of the forcing~$b_\kappa(t)$; the convolution encodes a memory of the forcing over a time scale of~$1/|\Re\mu|$.

\item \label{i:stsemud3}For unstable $\Re\mu>\muc$\,, and accepting anticipation in the transform (such as the second to last term in~\eqref{eq:stsetoyxyy}), modify the coordinate transform by setting $\fy=\Z\mu b_\kappa$\,, and do not change the $\Yv$-evolution, $\fg =0$\,.
    
\end{enumerate}
Consequently, this establishes equation~\cref{eq:stsesdenfyy} in \cref{thm:stsenfp}.

\subsection{Transform the rapid unstable dynamics}
\label{sec:sttud}

For the unstable dynamics, the argument corresponds directly to the argument of \cref{sec:sttfd} for the stable variables with appropriate exchange of symbols and inequalities.
This establishes equation~\cref{eq:stsesdenfzz}; that is, {we are always able to find a coordinate transform, to any specified order, which maintains the form~\cref{eq:stsesdenfzz}.}
Thus this subsection completes the proof of \cref{thm:stsenfp}.

\subsection{Centre dynamics do not anticipate}
\label{sec:sddnntatn}

Despite anticipatory convolutions often appearing in the coordinate transform~\cref{eq:stsexform}, this subsection establishes that, as in the example~\eqref{eqs:stsetoyXY}, no anticipation appears in the center dynamics because anticipatory convolutions always involve hyperbolic variables. 
\begin{full}
\cite{Bensoussan95} correspondingly showed it is not necessary to anticipate noise on a stochastic inertial manifold.
\end{full}

In the previous subsections, the anticipatory convolutions only occur when the rate~$\Re\mu>\muc$\,.
But for both the center and the hyperbolic components, this rate occurs only when at least one hyperbolic variable, $Y_j$ or~\(Z_k\), appears in the term under consideration.
Moreover, there is no ordinary algebraic operation that reduces the number of $\Yv$ and~\Zv\ factors in any term: 
potentially the time derivative operator might,
\begin{displaymath}
    \frac{d\ }{dt}=\D t{}
    +\sum_{\ell,k}X_kA_{\ell,k}\D {X_\ell}{}
    +\sum_{\ell,k}Y_kB_{\ell,k}\D {Y_\ell}{}
    +\sum_{\ell,k}Z_kC_{\ell,k}\D {Z_\ell}{}
    \,,
\end{displaymath}
but although in the algebra $X_\ell$~variables may be replaced by~$X_k$, the $Y_\ell$~variables may be replaced by~$Y_k$, and the $Z_\ell$~variables may be replaced by~$Z_k$, nonetheless the same number of variables are retained in each term and a $\Yv$ or~\Zv~variable is never replaced by an $\Xv$~variable.
The reason is that the center and hyperbolic dynamics are \emph{linearly} decoupled in the original system~\cref{eqs:stsesde}.
Consequently all anticipatory convolutions appear in terms with at least one component of the hyperbolic variables~$\Yv$ or~\Zv.

In essence, such anticipation, and also memory integrals, are a relic of the need by extant forward theory for time limits \(t\to\pm\infty\)\,.

Because of the form of the evolution~\cref{eq:stsesdenfxx} of the center modes~$\Xv$, the evolution~\cref{eq:stsesdenfxx} is also free of anticipatory convolutions within both the center-stable and center-unstable manifolds (\cref{def:nfnfim}).
However, as seen in examples, there may be anticipatory convolutions in the term \(\cF(t,\Xv,\Yv,\Zv)\Yv\Zv\) of~\cref{eq:stsesdenfxx}.
Further, although the coordinate transform~\cref{eq:stsexform} has anticipatory convolutions, on the center manifold $\Yv=\Zv=\ov $ there are none.
These algebraic deductions lead to the following corollary that extends straightforwardly that of \cite{Roberts06k} [Proposition~2].

\begin{corollary}
    \label{thm:stsemem}
Although \emph{anticipation} may be invoked,  throughout both  \(\Zv=\ov\) (the center-stable manifold, \cref{def:nfnfim}) and  \(\Yv=\ov\) (the center-unstable manifold), there \emph{need not} be any \emph{anticipation} in the dynamics~\cref{eq:stsesdenfxx} of the center modes in the non-autonomous normal form of the system~\cref{eqs:stsesde}.
Moreover, in $\Yv=\Zv=\ov$ (in the center manifold) the time 
dependent \emph{coordinate transform}~\cref{eq:stsexform} need not have anticipation.
\end{corollary}

\begin{full}
Nonetheless, despite the center manifold itself not displaying any anticipation, in general we need to anticipate the time dependence in the system in order to be \emph{always able to find a coordinate transform, to any specified order, which maintains a center $\Xv$~evolution that is independent of the hyperbolic variables throughout either the center-stable or the center-unstable manifold}, namely the \ode\ system~\cref{eq:stsesdenfxx}. 
Consequently, the projection of initial conditions, and the exponential approach to a solution of the center variables in the center-stable manifold, is assured only via invoking such anticipation in the full normal form coordinate transformation.
\end{full}

\begin{full}
\paragraph{Extension to rational functions forms}
Many biochemical systems, such as the Michaelis--Menten kinetics for enzyme dynamics, naturally arise in a rational function form.
\begin{corollary} \label{cor:}
Consider a non-autonomous system in rational function form
\begin{subequations}\label{eqs:stsesderat}%
\begin{align}
    \dot{\xv}&=\frac{A\xv+\fv(t,\xv,\yv,\zv)}{1+f(t,\xv,\yv,\zv)}\,,
    \label{eq:stsesdexrat}\\
    \dot{\yv}&=\frac{B\yv+\gv(t,\xv,\yv,\zv)}{1+g(t,\xv,\yv,\zv)}\,,
    \label{eq:stsesdeyrat}\\
    \dot{\zv}&=\frac{C\zv+\hv(t,\xv,\yv,\zv)}{1+h(t,\xv,\yv,\zv)}\,,
    \label{eq:stsesdezrat}
\end{align}
\end{subequations}
under \cref{ass:givensys} and additionally where scalar functions~\(f,g,h\) are~\(\cC^\fp(d)\) and~\Orc{1}.
Then \cref{thm:stsenfp} applies with the above system~\cref{eqs:stsesderat} replacing~\cref{eqs:stsesde}.
\end{corollary}
\begin{proof} 
The proof and arguments of \cref{sec:sttsd,sec:sttfd,sec:sttud,sec:sddnntatn} also apply to the system~\cref{eqs:stsesderat} when the residuals of the equations are obtained from the form
\begin{subequations}\label{eqs:stsesderas}%
\begin{align}
    \dot{\xv}&={A\xv+\fv(t,\xv,\yv,\zv)-\dot{\xv}f(t,\xv,\yv,\zv)},
    \label{eq:stsesdexras}\\
    \dot{\yv}&={B\yv+\gv(t,\xv,\yv,\zv)-\dot{\yv}g(t,\xv,\yv,\zv)},
    \label{eq:stsesdeyras}\\
    \dot{\zv}&={C\zv+\hv(t,\xv,\yv,\zv)-\dot{\zv}h(t,\xv,\yv,\zv)}.
    \label{eq:stsesdezras}
\end{align}
\end{subequations} 
\end{proof}
Alternatively, one could also point out that the rational function form~\eqref{eqs:stsesderat} may be written in the form~\cref{eqs:stsesde}: such as 
\(\dot\xv=A\xv+\big(\frac{\fv-fA\xv}{1+f}\big)\).
However, in application the form~\cref{eqs:stsesderas} provides a more practical route for construction \cite[e.g.,][]{Roberts09c}. 
\end{full}

\subsection{Alternative slow-fast subcenter separation}
\label{sec:asfss}

Many modelling scenarios require the separation of fast waves from interesting slow dynamics: for example, elasticity \cite[e.g.,][]{Muncaster83c, Cohen88}, quasi-geostrophy \cite[e.g.,][]{Leith80, Lorenz86, Lorenz87, Warn95}, anelastic approximation \cite[e.g.,][]{Durran89}, incompressible fluid flow \cite[Ch.~13, e.g.]{Roberts2014a}, and many physics scenarios \cite[Parts~II and~III, e.g.]{vanKampen85}.
As indicated in \cref{eg:qga}, the quasi-geostrophic slow manifold apparently does not exist, and yet the slow manifold is an essential component of the computer code that, every hour of every day everywhere around the earth, makes reliable weather forecasts \cite[e.g.]{Roulstone2013}!
Our backward approach provides new insight into this apparent paradox.

In the finite-D systems considered herein, the issue is the separation of fast oscillations from interesting slow dynamics.
\begin{full}
Indeed, the following lemma establishes that dissipative problems with a spectral gap separating slow-center modes and stable modes may be often linked to `mechanical' problems with a spectral gap separating slow and fast waves.

\begin{lemma}
In linear autonomous systems for~\(u(t)\),
\(\dot u=\cL u\) has invariant subspace \(u=\cP U\) such that \(\dot U=\cA U\) iff
\(\ddot u=\cL u\) has invariant subspace \(u=\cP U\) such that \(\ddot U=\cA U\).
\end{lemma}
\begin{proof}
Straightforwardly follows from solutions of the generalised eigen-problem \(\cL\cP=\cP\cA\).
\end{proof}

That is, in first order linear systems \(\dot u=\cL u\), every real negative eigenvalue~\(\lambda\) gives rise to the pure imaginary eigenvalues~\(\pm\sqrt\lambda\) in the corresponding `mechanical' problem \(\ddot u=\cL u\)\,.  
Thus, if there is a suitable spectral gap in the real eigenvalues of~\cL, then the corresponding `mechanical' problem has a spectral gap separating fast waves from slow waves.
\end{full}
In such scenarios the preceding derivations of \cref{sec:stsenfsmsde} usefully apply with some changes.  
Indeed, the use of a coordinate transform to understand such slow-fast separation of oscillations---albeit for autonomous systems---dates back at least sixty years to the work of \cite{Kruskal1962}.  
But the backward \cref{pro:sccte} proved in this section is unprecedented \text{in extant theory.}

Recall that, for a chosen threshold~\muc,  \cref{def:ocf} terms quantities ``fast'' when the absolute value \(|\lambda|>\muc\) and ``slow'' when \(|\lambda|\leq\muc\)\,.  
Then the algebraic derivations of \cref{sec:sttsd,sec:sttfd} applies with \(|\Re\mu|\) replaced by~\(|\mu|\),  ``hyperbolic'' replaced by ``fast'', and ``center'' replaced by ``slow''.

However, there are two significant changes.
There is a qualitative change to the variables and dynamics.
In the corresponding normal form~\cref{eqs:stsesdenf} there are no~\Zv\ variables as here the fast~\Yv\ encompasses all hyperbolic variables.
Correspondingly the evolution on the slow manifold has the generic form \(\dot\Xv=A\Xv+\Fv_c(t,\Xv)+\cF(t,\Xv,\Yv)\Yv\Yv\). 
This form reflects that there typically are quadratic interactions between fast modes\slash waves, interactions that affect the evolution of the slow variables.
One example of such quadratic interactions is the Stokes drift mean flow generated by water waves \cite[p.425, e.g.]{Mei89}.
The important result is the following: \emph{in the presence of fast waves\slash oscillations the evolution off such a slow manifold is typically fundamentally different to that on the slow manifold.}

The convolutions~\eqref{eq:stsezmuf} must be adapted to cater for `large' frequency oscillations: \(\mu\)~satisfying small~\(|\Re\mu|\) and \(|\mu|>\muc\).
Recall that the convolutions are only needed to solve \ode{}s of the form~\eqref{eq:stseddtconv}, so in this scenario define the symbol~\(\Z\mu a\) to denote, if one exists, a bounded solution of~\eqref{eq:stseddtconv}---a bound uniform over the time interval~\(\TT_\mu\).
Then to solve the analogue of~\eqref{eq:stsemux}, if \(\Z\mu a_\kappa\)~exists, then set \(\fx=\Z\mu a_\kappa\) and \(\ff=0\)\,, otherwise set \(\fx=0\) and \(\ff=a_\kappa\).
Similarly, to solve the analogue of~\eqref{eq:stsemuy}, if \(\Z\mu b_\kappa\)~exists, then set \(\fy=\Z\mu b_\kappa\) and \(\fg=0\)\,, otherwise set \(\fy=0\) and \(\fg=b_\kappa\).
For some right-hand sides, judicious integration by parts may be used to maintain some subjectively desirable properties\slash symmetries in the conjugate system.

Let's briefly discuss what terms may be required to be in the evolution of the slow variables~\Xv\ via the variable~\ff.
Analogous to \cref{sec:sttsd}, unbounded solutions~\fx\ of the \ode{}s arise when~\(a_\kappa\) is slowly varying (as fast oscillations effectively cancel, and/or fast growth\slash decay may be confined to negligible transients). 
Hence if~\(a_\kappa\) arises from terms either purely in~\Xv, or `resonance' terms involving two or more \Yv-variables, then we must generally avoid unboundedness by setting~\ff\ non-zero (similar considerations usefully transform stochastic Hopf bifurcations \cite[\S5, e.g.]{Roberts06k}).
Hence we chose~\ff\ to maintain the form \(\dot\Xv=A\Xv+\Fv_c(t,\Xv)+\cF(t,\Xv,\Yv)\Yv\Yv\).

The above considerations then establish the following proposition on the separation of slow and fast variables (analogous to \cref{thm:stsenfp}).

\begin{proposition}[fast-slow separation] \label{pro:sccte}
Consider the generic system~\cref{eqs:stsesde} in the case where there are no \zv-variables, the \xv-variables are slow,  \(|\alpha_i|\leq\alpha<\muc\), the \yv-variables are fast, \(|\beta_j|\geq\beta>\muc\) (\cref{def:ocf}), and the explicit non-autonomous effects are `slow' in time.
For every given order~\(p\), \(2\leq p\leq\fp\)\,, there exists a near identity, polynomial, time dependent, coordinate transformation
\begin{equation}
    \xv=\xv(t,\Xv,\Yv),\quad
    \yv=\yv(t,\Xv,\Yv),
    \label{eq:scctexform}
\end{equation}
and a corresponding polynomial normal form system
\begin{subequations}\label{eqs:sccte}%
\begin{align}
    \dot{\Xv}&=A\Xv+\Fv (t,\Xv,\Yv)
    \nonumber\\
    &=A\Xv+\Fv_c(t,\Xv)+\cF(t,\Xv,\Yv)\Yv\Yv,
     \label{eq:scctexx}\\
    \dot{\Yv}&=\big[ B+G(t,\Xv,\Yv) \big]\Yv,
    \label{eq:sccteyy}
\end{align}
\end{subequations}
which together is approximately conjugate to the non-autonomous system~\cref{eqs:stsesde},
where~\Fv\ and~\(G\Yv\) are~\Orc{2}  (and where \cF~is a rank three tensor),
and where, by construction, the difference between system~\cref{eq:scctexform}+\cref{eqs:sccte} and system~\cref{eqs:stsesde} is~\Orc{p}.
\end{proposition}

Systems in the normal form~\eqref{eqs:sccte} possess the slow manifold \(\Yv=\ov\) (\cref{def:nfnfim}) on which the slow evolution is \(\dot{\Xv}=A\Xv+\Fv_c(t,\Xv)\).
Consequently, an immediate corollary of \cref{pro:sccte} is that there are systems \(\Orc{p}\)~close to~\cref{eqs:stsesde} which possess a slow manifold, in some domain about~\ov, and parametrised by \(\xv=\xv(t,\Xv,\ov)\) and \(\yv=\yv(t,\Xv,\ov)\).

\begin{full}
\paragraph{Nonlinear normal modes} Without elaborating details, I conjecture that there should be cognate propositions about the existence of coordinate transforms establishing that for every given suitable system there exists a nearby system which possesses an exact nonlinear normal mode \cite[cf.][e.g.]{Shaw94a, Haller2016}.
\end{full}

In contrast, forward theory of invariant manifolds says almost nothing about the existence and relevance of such slow manifolds, \begin{full}
nor about nonlinear normal modes, \end{full}even in the simpler case of non-autonomous scenarios.
For example, the subcenter Theorem~7.1 by \cite{Sijbrand85} does not apply to the crucial notion of quasi-geostrophy in geophysical fluid dynamics \cite[e.g.,][Ch.~7]{Gill82}.
Of course, part of the reason for this lack of general theory is the extreme sensitivity of topological structures when there are slow dynamics among fast oscillations.  
But firstly, preserving known symmetries reduces the scope for such sensitivity.
And secondly, when one is mainly interested in finite time, such as a weather forecast for the next week, then such topological sensitivity may well be of negligible relevance, or mey be catered for by ensemble forecasts \cite[e.g.,][]{Palmer05, Roulstone2013}.
The following \cref{eg:qga} illustrates how the backward approach illuminates the slow manifold of quasi-geostrophy---an enormously important practical concept in geophysical dynamics \cite[e.g.,][]{Leith80}.

\begin{example} \label{eg:qga}
In order to explore and understand the slow manifold of quasi-geostrophy, \cite{Lorenz86} introduced the toy 5-D autonomous system, often called the \emph{Lorenz86 system},
\begin{align}&
	\dot u  =  -vw+bvz\,,  \quad
	\dot v  =  \phantom{-}uw-buz\,,  \quad
	\dot w  =  -uv\,,  \nonumber\\&
	\dot x  =  -z\,, \quad
	\dot z  =  \phantom{-}x+buv\,. \label{eqs:qga}
\end{align}
In the algebraic form of system~\eqref{eqs:qga} there is a clear cut distinction between the slow modes~$u$, $v$, and~$w$ (our~\xv), and the fast modes~$x$ and~$z$ (our~\yv), a distinction which \cite{Lorenz86} and others use to construct a `slow manifold'.
But \cite{Lorenz87} then proved that there is no slow manifold for the system~\eqref{eqs:qga}!  
This almost paradoxical result, attributable to a lack of persistent topological structures, led to much discussion \cite[e.g.,][]{Jacobs91, Lorenz92, Camassa95, Fowler96, Vanneste2008, Ginoux2013}.
Nonetheless, geophysics, science, and engineering need to use such slow manifolds.

In our backward approach, \cref{thm:stsenfp} immediately applies with the above mentioned modifications.
Indeed, \cite{Cox92} constructed such a normal form---the algebraic construction described by \cite{Cox93a} corresponds to those detailed in \cref{sec:stsenfsmsde}, although the algebraic machinations are considerably simpler because the Lorenz86 system~\eqref{eqs:qga} is autonomous.
To third order of asymptotic approximation, differences~\Ord{b^3} as \(b\to0\), the near identity coordinate transformation
\begin{subequations}\label{eqs:nfrl1}%
\begin{align}&
	u = U -bVX +\tfrac14b^2U(Z^2-X^2),  \quad
	v =  V +bUX +\tfrac14b^2V(Z^2-X^2) ,  \nonumber\\&
	w =  W +b^2Z(V^2-U^2),  \quad
	x =  X -bUV +\tfrac14b^2X(V^2-U^2),  \nonumber\\&
	z =  Z -bW(V^2-U^2)-\tfrac14b^2Z(V^2-U^2), \label{eq:nfrl1x}
\end{align}
and evolution equations
\begin{align}&
	\dot U = - VW -b^2VW(V^2-U^2),  \quad
	\dot V = UW +b^2UW(V^2-U^2),  \nonumber\\&
	\dot W = -UV +{\color{red}\tfrac12b^2UV(X^2+Z^2)},  \quad
	\dot X = -Z +\tfrac12b^2Z(V^2-U^2),  \nonumber\\&
	\dot Z =  X -\tfrac12b^2X(V^2-U^2)\,,  \label{eq:nfrl1e}
\end{align}
\end{subequations}
together, are approximately the Lorenz86 system~\cref{eqs:qga} \cite[]{Roberts2012a}.
\begin{itemize}
\item \cref{def:nfnfim} then implies that \(X=Z=0\) is a slow manifold for~\cref{eq:nfrl1e}, and consequently \(u=U\), \(v=V\), \(w=W\), \(x=-bUV\) and \(Z=-bW(V^2-U^2)\) is an exact 3-D slow manifold for \cref{eqs:nfrl1} parametrised by~\((U,V,W)\).
Although a slow manifold does not exist for the original~\cref{eqs:qga} \cite[]{Lorenz87}, evidently there are systems exponentially close%
\footnote{The exponential closeness (\cref{def:ocf}) follows since we may continue the construction of systems like~\cref{eqs:nfrl1} to arbitrarily high order (\cref{cor:nfesc}).} 
to~\cref{eqs:qga} which do possess a slow manifold (corresponding \text{to \cref{thm:nfimex}).}

\item An extension of \cref{lem:nfftd} would establish a lower bound on the size of the domain of existence for the slow manifold of~\cref{eq:nfrl1e}, and thence for~\cref{eqs:nfrl1} (\cref{def:nfdd}).

\item On the slow manifold of~\cref{eqs:nfrl1}, the evolution is given by
\cref{eq:nfrl1x} with \(X=Z=0\) (corresponding to \cref{cor:nfeom}), namely \(\dot U = - VW -b^2VW(V^2-U^2)\), \(\dot V = UW +b^2UW(V^2-U^2)\), and \(\dot W = -UV\)\,.

\item Further, this backward approach illuminates why there can be no general analogue of the emergence \cref{lem:nfemerge} in slow manifold scenarios like the Lorenz86 system~\cref{eqs:qga}, and hence there is here no counterpart of \cref{pro:stbackem}.  
Although the approximate system~\cref{eqs:nfrl1} has a slow manifold, the slow manifold is not attractive.  
Instead, the conjugate system~\cref{eq:nfrl1e} indicates that nearby solutions, nonzero~\((X,Z)\), oscillate indefinitely.  
But further, such oscillations generally force a drift in the slow variables: in~\cref{eq:nfrl1e} the red-coloured quartic term in~\(\dot W\) generates a drift proportional to the square of the fast-wave amplitude (\(X^2+Z^2\)), a drift that is not present on the slow manifold. 
Consequently, the evolution on such a slow manifold cannot be expected to exactly capture all the nearby evolution---even in an average sense.

\end{itemize}
It seems apparent that the above discussion for the example Lorenz86 system is quite generic in separating slow-fast subcenter dynamics, and clarifies important issues in geophysics, science, and engineering dynamics.

\end{example}

\section{Existence and emergence proved in a domain}
\label{sec:nfeep}

\begin{full}
\begin{quoted}{Richard E. Meyer, 1992}
All differential equations are imperfect models 
\end{quoted}
\end{full}

Our theory is based upon the dynamics near an equilibrium (at the origin without loss of generality).  
In that sense the approach is local.  
Nonetheless, by continuity in various bounds, the locale of theoretical support is finite in size.  
This section establishes the first general lower bound on the finite domain of validity.
It is well known that in some scenarios extant forward establishes a global domain of validity, but it can only do so for scenarios with strong restrictions on the nonlinearity and the time-dependence, scenarios rarely of use in applications.
The lower bound here quantifies that the center manifold framework is not ``just asymptotic'', the theory supports use at the finite parameter values that are essential in most applications, and for nonlinearities that are usual in applications.

\begin{full}
A key part of our dynamical systems approach are the existence and properties of invariant\slash integral manifolds of a given non-autonomous \ode\ system.
Future research plans to cater for general stochastic systems where  very rare events, of mostly negligible practical interest, could send trajectories outside a domain of validity.
Such very rare events severely limit the usefulness of established forward theory.
\end{full}
In any specified system~\cref{eq:phys}, with linear operator~\cL, most  previous definitions of invariant\slash integral manifolds require \(e^{\cL t}\) to be analysable in both forward and backward time, that is, the operator~\cL\ must be bounded. 
But in extensions of our approach \cite[]{Hochs2019} to \pde{}s the operator~\cL\ is typically unbounded.
Consequently we change some basic definitions to cope: we replace current extant definitions of invariant\slash integral manifolds with the following \cref{def:nfnfim}.

In \cref{def:nfnfim} the subspaces~\(\MM_i\) are invariants of~\cref{eqs:stsesdenf} across the whole \(\Uv=(\Xv,\Yv,\Zv)\)-space. 
But their characterisation with regard to the dynamics about the equilibrium is only certain within a finite domain, denoted~\(\DD_\mu\).
So we restrict the manifolds to that domain (and then avoid calling them subspaces).
Moreover, because the domain~\(\DD_\mu\) may only exist for a finite time interval~\(\TT_\mu\) the established definitions of invariant manifolds \text{do not apply here.}

\begin{definition}[invariant\slash integral manifolds] \label{def:nfnfim}
First, for every system in the normal form~\cref{eqs:stsesdenf}, and for time intervals~\(\TT_\mu\) and domains~\(\DD_\mu\) containing the origin and characterised by \cref{def:nftpd}, define the following \emph{invariant\slash integral manifold}s associated with the equilibrium at the origin:%
\footnote{These definitions are for ``\emph{a \ldots\ manifold}'': the implied non-uniqueness here is due to potentially differing~\(\mu\), \(\TT_\mu\) and~\(\DD_\mu\).
The potential for exponentially small non-uniqueness allowed by previous (forward theory) definitions of invariant\slash integral manifolds is excluded by \cref{def:nfnfim} for every system in the  form~\cref{eq:stsexform}+\cref{eqs:stsesdenf}.}
\begin{itemize}
\item a \emph{center manifold} \(\MM_{c}\) is 
\((\Xv,\ov ,\ov )\in\DD_\mu\);
\item a \emph{stable manifold} \(\MM_{s}\) is 
\((\ov ,\Yv,\ov )\in\DD_\mu\);
\item an \emph{unstable manifold} \(\MM_{u}\) is 
\((\ov ,\ov ,\Zv)\in\DD_\mu\);
\item a \emph{center-stable manifold} \(\MM_{cs}\) is 
\((\Xv,\Yv,\ov )\in\DD_\mu\);
\item a \emph{center-unstable manifold} \(\MM_{cu}\) is 
\((\Xv,\ov ,\Zv)\in\DD_\mu\).
\end{itemize}
Further, in the extension of \cref{sec:asfss} to slow-fast variables~\((\Xv,\Yv)\), define a \emph{slow manifold}~\(\MM_{0}\) to be \((\Xv ,\ov )\in\DD_\mu\).
Lastly, denoting more concisely the \(\cC^\fp\)-diffeomorphism~\cref{eq:stsexform} from some \(D_\mu\subseteq\DD_\mu\) onto~\(d_\mu\) as \(\uv=\uv(t,\Uv)\), for every system~\cref{eq:stsexform}+\cref{eqs:stsesdenf} and for every \(i\in\{c,s,u,cs,cu,0\}\),  define corresponding \emph{invariant\slash integral manifolds}~\(\cM_i(t)=\{\uv(t,\MM_i\cap D_\mu): t\in\TT_\mu\}\).
\end{definition}

In contrast, established forward theory for non-autonomous systems defines invariant manifolds in terms of integrals over infinite time.%
\footnote{Some example preconditions in extant theory are the following.
\cite{Henry81} in Defn.~6.1.1 requires time over all~\RR.
\cite{Potzsche2006} [p.431] invoke ``\(\mathbb I\) denotes a real [time] interval unbounded above'' with [p.438] ``under the additional assumption that \(\mathbb I=\RR\)'' that unstable manifolds exist.  
\cite{Bento2017b} in their non-uniform analysis similarly require ``for all \(t\in \RR\)'' in their Theorem~3.1.
\cite{Barreira2007} 
[p.172] require the nonlinearity function in the system to decay exponentially quickly in time, ``for every \(t\in\RR\)''.
And for their non-autonomous invariant manifolds, \cite{Aulbach2006} require ``for all \(t,\tau\in\RR\)'' in their definition of an invariant set~[p.3], and similar elsewhere.
Such preconditions are often unavailable in applications.
}
 
Consequently forward theory requires properties of the system to hold for all time, and the trajectories of the system to be suitable for all time.  
Such \emph{all time} requirements are too onerous for many applications (e.g., rare events may `kick' a system to a completely different domain in state space). 
The above new definitions do not invoke infinite time intervals.

The well known non-uniqueness of invariant manifolds appears here in the non-uniqueness of the coordinate diffeomorphism~\cref{eq:stsexform}.
For a classic example, the 2D system \(\dot x=-x^3\) and \(\dot y=-2y\) is well known to have center manifolds \(y=c\exp(-1/x^2)\) for every~\(c\). 
But the coordinate transform \(X=x\) and \(Y=y-c\exp(-1/x^2)\) leads to \(\dot X=-X^3\) and \(\dot Y=-2Y\) (symbolically the same system) which \cref{def:nfnfim} asserts has in~\((X,Y)\) only the center manifold \(Y=0\)\,.
That is, for each diffeomorphism~\cref{eq:stsexform} a unique invariant manifold is defined. 
That the constant~\(c\) is arbitrary in this example reflects that non-uniqueness arises because there are many suitable diffeomorphisms~\cref{eq:stsexform}.

Sometimes the domain~\(D_\mu\) of theoretical support is `infinite' in size.
\begin{full}
Such as in the example of \cref{eg:nfect}.
\end{full}
One important class of global support is when there is a whole subspace\slash manifold of equilibria (e.g., see \cref{sec:nfssppei}), each of which satisfies the criteria for the theory.
In such scenarios, such a collection of such local invariant manifolds forms a so-called \emph{global} invariant manifold \cite[e.g.,][]{Carr81} because it exists in a domain containing all the equilibria, albeit maybe only local in directions transversal to the manifold of equilibria.

Variations to the Hartman--Grobman Theorem for {Caratheodory}-type differential equations \cite[e.g.,][]{Aulbach2000} ensure this \cref{def:nfnfim} is consistent, where both apply, with extant definitions \cite[e.g.,][]{Henry81, Barreira2007, Haragus2011}.
\cite{Henry81} refers to finite-time invariant manifolds as ``local invariant manifolds'' [Defn.~6.1.1], while others reserve this term to mean local in state space \cite[e.g.,][]{Aulbach2000}, whereas \cref{def:nfnfim} proposes the practical view that finite-time intervals and finite state-space domains are the norm in applications.
Potentially there are cases for which we know invariant manifolds exist under extant definitions, but not knowably under this \cref{def:nfnfim}.
This new backward theory complements extant forward theory for the many applications where previous definitions do not apply (e.g., \cref{eg:qga}) but for which this \cref{def:nfnfim} knowably characterises useful invariant manifolds.

\subsection{Time scales remain separated in a domain}

For straightforward use of a multivariate Lagrange's remainder theorem, this section restricts domains to be `star-shaped' relative to the base equilibrium.  
That is, each point in a domain must be connected to the equilibrium (assumed to be at the origin) by a straight line segment that stays within the domain.
To relax this star-shaped constraint, one would adapt the use of the multivariate Lagrange remainder theorem. 

\begin{definition}[star-shaped] \label{def:nfss}
A non-empty open subset \(S\subseteq\RR^N\) is \emph{star-shaped} if for every \(\uv \in S\), \(\gamma\uv \in S\) for all \(0\leq\gamma\leq1\).
\end{definition}

\begin{full}
In \cref{fig:ctECT} for the example of \cref{eg:nfect}, the left-hand \(xy\)-domain (shaded) is star-shaped, but the right-hand \(XY\)-domain is not (and would have to be restricted a little for the results developed here).
\end{full}

The major results of this section are that key invariant manifolds exist and emerge from the dynamics over a domain whose size we bound from below.
Crucial modelling properties often hold over domains usefully larger than~\(\DD_\mu\) that may be identified in any given scenario.
The size of the domain~\(\DD_\mu\) is primarily bounded by the typical increase from zero, of~\(\|G(t,\Uv)\|\) and~\(\|H(t,\Uv)\|\) (in~\cref{eqs:stsesdenf}) as \(|\Uv|\)~increases, where \(\Uv=(\Xv,\Yv,\Zv)\in\RR^{m+n+\ell}\).
As usual, the norm~\(\|\cdot\|\)~denotes the \(2\)-norm of the matrix, viz \(\|G\|:=\max_{|\vec v|=1}|G\vec v|\), and also the condition number \(\cond U:=\|U^{-1}\|\cdot\|U\|\).

\begin{definition}[emergence preserving domain] \label{def:nftpd}
Consider any given polynomial normal form system~\cref{eqs:stsesdenf}. 
Define matrices~\(P\), \(Q\) and~\(R\) to be such that~\(P^{-1}AP\), \(Q^{-1}BQ\) and~\(R^{-1}CR\) are in (real) Jordan form, and define~\(\delta\geq0\) to be the maximum magnitude of the super-diagonal elements in these Jordan form matrices.%
\footnote{We seek matrices~\(P\), \(Q\) and~\(R\) that reduce matrices~\(A\), \(B\) and~\(C\) to block-diagonal form where here the term `diagonal' includes \(2\times 2\)~blocks of the real form \(\left[\begin{smallmatrix}  \Re& \Im\\ -\Im& \Re \end{smallmatrix}\right]\) for each pair of complex conjugate eigenvalues~\(\Re\pm i\Im\). 
So here ``super-diagonal'' does not include these \(2\times 2\)~blocks.
Further, the Jordan blocks for degenerate eigenvalues have super-diagonal elements of at most~\(\delta\) which can be as small as desired at the trade-off of increasing (worsening) the condition number of~\(P\), \(Q\) and~\(R\).}
Then define a star-shaped domain \(\DD_\mu\subseteq \RR^{m+n+\ell}\) and time interval~\(\TT_\mu\), such that the parameter~\(\mu\) is  within the spectral gap  \(\alpha<\mu\pm\delta<\beta-\max\big\{\cond(Q)\|{G}(t,\Uv)\|\,,\linebreak[0]\,\cond(R)\|{H}(t,\Uv)\|\big\}\) for all \(\Uv\in\DD_\mu\) and for all times \(t\in \TT_\mu\)\,.%
\end{definition}

The domain \(\DD_\mu\) is non-trivial for the following reasons.
Firstly, \(G\Yv,H\Zv\) are~\Orc2\ so the origin is in~\(\DD_\mu\) for all time.
Secondly, continuity of the polynomial~\(G\Yv\) and~\(H\Zv\) assures us that \(\DD_\mu\)~is a finite neighbourhood about the origin (the neighbourhood may be non-uniform in time, so we restrict consideration to some time interval~\(\TT_\mu\) that contains the initial time~\(t_0\)). 
The above \cref{def:nftpd}, for the domain~\(\DD_\mu\), only addresses the class of normal form systems~\cref{eqs:stsesdenf}:  the algebraic machinations of \cref{sec:stsenfsmsde} derive the system~\cref{eqs:stsesdenf} as part of an approximation, but its role as an approximation is not relevant in this definition.

\begin{lemma}[exponential trichotomy]\label{lem:exptri}
For times \(s,t\in\TT_\mu\) and for as long as solutions of~\cref{eqs:stsesdenf} stay in~\(\DD_\mu\):  
\begin{itemize}
\item \(|\Yv(t)|\leq \cond Q\,|\Yv(s)|e^{-\mu (t-s)}\) for \(t\geq s\)\,; 
\item \(|\Zv(t)|\leq \cond R\,|\Zv(s)|e^{-\mu (s-t)}\) for \(t\leq s\)\,; 
\item \(|\Xv(t)|\leq \cond P\,|\Xv(s)|e^{\mu|t-s|}\) provided \(\Yv(s)=\ov\) or \(\Zv(s)=\ov\).
\end{itemize}
\end{lemma}

\begin{proof} 
First establish the bound \(|\Yv(t)|\leq \cond Q\,|\Yv(s)|e^{-\mu (t-s)}\).
The case of initial state $\Yv(s)=\ov $ is trivial, so consider $\Yv(s)\neq \ov $\,.   
As suggested by Chicone (private communication, 2013), a Lyapunov function establishes that $\Yv=\ov $ is exponentially quickly attractive.  
Define the (Lyapunov function) vector norm \(|\Yv|_Q:=|Q^{-1}\Yv|\) in terms of the usual \(2\)-norm \(|\xv|:=\sqrt{\tr{\xv}\xv}\), and in terms of the matrix~\(Q\) that transforms matrix~\(B\) (\cref{def:nftpd}).
Let \(J:=Q^{-1}BQ\), then straightforward algebra derives that
\begin{equation}
\de t{}|\Yv|_Q^2=\tr{(Q^{-1}\Yv)}\left[
J+\tr J+(Q^{-1}GQ)+\tr{(Q^{-1}GQ)} \right](Q^{-1}\Yv).
\label{eq:nfyvbnd}
\end{equation}
Consider two parts of the right-hand side in turn.
\begin{itemize}
\item From the real Jordan form (\cref{def:nftpd}), then symmetric \(\tr J+J=2\diag(\Re \beta_j)+\Delta\) for eigenvalues~\(\beta_j\) of~\(B\) and symmetric \(\Delta\) being zero except for some sub/super-diagonal elements of at most~\(\delta\).
Consequently the quadratic
\begin{eqnarray*}
\tr{(Q^{-1}\Yv)}\left[J+\tr J\right](Q^{-1}\Yv)
&\leq& |Q^{-1}\Yv|^2\max \lambda(J+\tr J)
\\&\leq&|\Yv|_Q^2\max_j (2\Re\beta_j\pm2\delta)
\\&<&-2(\beta-\delta) |\Yv|_Q^2.
\end{eqnarray*}

\item The remaining part of the right-hand side is usefully bounded since
\begin{eqnarray*}
&&\left|\tr{(Q^{-1}\Yv)}\left[(Q^{-1}GQ)+\tr{(Q^{-1}GQ)} \right](Q^{-1}\Yv)\right|
\\&&{}
\leq 2|Q^{-1}\Yv|\cdot\|Q^{-1}GQ\|\cdot|Q^{-1}\Yv|
\\&&{}
\leq 2\|Q^{-1}\|\cdot\|G\|\cdot\|Q\|\cdot|Q^{-1}\Yv|^2
= 2\cond Q\,\|G\|\cdot|\Yv|_Q^2\,.
\end{eqnarray*}
\end{itemize}
Consequently, in the domain~\(\DD_\mu\) equation~\cref{eq:nfyvbnd} ensures the inequality
\begin{equation*}
\de t{}|\Yv|_Q^2 <2\big(-\beta+\delta+\cond Q\,\|G\|\big)|\Yv|_Q^2
<-2\mu|\Yv|_Q^2\,.
\end{equation*}
Standard comparison theorems then ensure \(|\Yv(t)|_Q^2\leq |\Yv(s)|_Q^2e^{-2\mu (t-s)}\); that is, \(|\Yv(t)|_Q\leq |\Yv(s)|_Qe^{-\mu (t-s)}\).  
Thus the stable variables~\(\Yv\) decay to zero exponentially quickly in the \(Q\)-norm.
In the usual \(2\)-norm there may be some transient growth characterised by the condition number of~\(Q\) \cite[e.g.,][]{Trefethen99}.  
Returning to the \(2\)-norm, we derive the bound
\begin{eqnarray*}
|\Yv(t)|&=&|QQ^{-1}\Yv(t)|
\\&\leq& \|Q\|\cdot|Q^{-1}\Yv(t)|
=\|Q\|\cdot|\Yv(t)|_Q
\\&\leq& \|Q\|\cdot|\Yv(s)|_Qe^{-\mu (t-s)}
= \|Q\|\cdot|Q^{-1} \Yv(s)|e^{-\mu(t-s)}
\\&\leq& \|Q\|\cdot\|Q^{-1}\|\cdot| \Yv(s)|e^{-\mu(t-s)}
= \cond Q\,| \Yv(s)|e^{-\mu(t-s)}
\end{eqnarray*}
for as long as solutions stay in the domain~\(\DD_\mu\) \cite[cf.][Lemma~5.3.1]{Murdock03}.

Second, a proof for the bound \(|\Zv(s)|\leq \cond R\,|\Zv(t)|e^{-\mu (t-s)}\) is the corresponding argument but backward in time.

Third, establish the bound for~\Xv.
Under the proviso that one of \(\Yv(s)=\ov\) or \(\Zv(s)=\ov\), and hence by~\cref{eqs:stsesdenf} is zero for all \(t\in\TT_\mu\).
By the form of~\cref{eq:stsesdenfxx}, then \(\Fv(t,\Xv,\Yv,\Zv)=\Fv(t,\Xv,\ov,\ov)\).
Since \(\DD_\mu\) is star-shaped, by Lagrange's remainder theorem
\begin{equation*}
\Fv(t,\Xv,\ov,\ov)={\Fv(t,\ov,\ov,\ov)}+F(t,\Xv)\Xv=F(t,\Xv)\Xv
\end{equation*}
for matrix \(F(t,\Xv):=\Fv_\Xv(t,\xi\Xv,\ov,\ov)\) for some \(0\leq\xi(t,\Xv)\leq1\)\,.
Then by corresponding arguments to those for~\Yv,
\begin{equation*}
\de t{}|\Xv|_P^2 <2\big(\alpha+\delta+\cond P\,\|F\|\big)|\Xv|_P^2
<2\mu|\Xv|_P^2\,,
\end{equation*}
which ensures \(|\Xv(t)|_P^2\leq|\Xv(s)|_P^2e^{2\mu(t-s)}\) for \(t\geq s\)\,.
Similarly, in integrating backward in time
\begin{equation*}
-\de t{}|\Xv|_P^2 <2\big(\alpha+\delta+\cond P\,\|F\|\big)|\Xv|_P^2
<2\mu|\Xv|_P^2\,,
\end{equation*}
so \(|\Xv(t)|_P^2\leq|\Xv(s)|_P^2e^{2\mu|t-s|}\).
Then by an argument corresponding to that for~\Yv, we deduce the 2-norm bound
\begin{equation*}
|\Xv(t)|\leq\cond P\,|\Xv(s)|e^{\mu|t-s|}
\end{equation*}
for as long as solutions stay in~\(\DD_\mu\) and provided \(|\Yv(s)|\cdot|\Zv(s)|=0\).
\end{proof}

\cref{def:nftpd} assures us that for the normal form~\cref{eqs:stsesdenf} the important center-stable manifold  exists in a finite domain.
The following \cref{lem:nfftd} establishes one minimum bound on the size of the finite domain.
Cognate results apply to other invariant manifolds.

\begin{lemma}[finite domain] \label{lem:nfftd}
Restrict attention to the non-empty center-stable manifold,~\(\MM_{cs}\), of the polynomial, normal form, system~\cref{eqs:stsesdenf}.
The center-stable manifold~$\MM_{cs}$ in~\(\DD_\mu\) contains the ball~\(B_\mu\) in~\(\MM_{cs}\) centred on the origin and of radius $(\beta-\mu-\delta)^2/\big[2\cond^2(Q)G'_{\max}\big]$,  where  \(G'_{\max}=\sup_{t\in\TT_\mu,\Uv\in B_\mu}\big|\sum_{i,j}g_{ij}\grad g_{ij}\big|\) in terms of the elements of \(G=\begin{bmatrix} g_{ij}(t,\Uv) \end{bmatrix}\).
\end{lemma}

\begin{full}
\begin{example} 
Before proving \cref{lem:nfftd}, let's apply it to the 2D example of \cref{eg:nfect}. 
The \(Y\)-evolution~\cref{eq:cmECT} has \(G=g_{11} =1-1/(1+2X^2)-4X^2 =2X^2/(1+2X^2)-4X^2\).  
Then 
\begin{eqnarray*}
g_{11}\grad g_{11}
&=&\left[\frac{2X^2}{1+2X^2}-4X^2\right]
\left[\frac{4X}{(1+2X^2)^2}-8X\right]
\\&=&8X^3\left[2-\frac{1}{1+2X^2}\right]
\left[2-\frac{1}{(1+2X^2)^2}\right]
\\
\implies |g_{11}\grad g_{11}|&\leq&32|X|^3.
\end{eqnarray*}
Thus, in \(|X|<r\) this identity gives \(|g_{11}\grad g_{11}| <G'_{\max}\leq 32r^3 \).
Since here \(\beta=\cond Q=1\) and \(\delta=0\)\,, the radius~\(r\) of the ball~\(B_\mu\) in the lemma requires \(r<(1-\mu)^2/(64r^3)\); that is, \(r<\sqrt{1-\mu}/(2\sqrt2)\). 
This \cref{lem:nfftd} certifies that the shaded domain of \cref{fig:ctECT}~(right) is large enough to contain these balls of radius~\(\sqrt{1-\mu}/(2\sqrt2)\).
But further, \cref{lem:exptri} guarantees, for every \(0<\mu<1\)\,, that solutions decay to \(Y=0\) at least as fast as~\(e^{-\mu t}\) while inside~\(B_\mu\).
\end{example}
\end{full}

\begin{example} 
Before proving \cref{lem:nfftd}, let's apply it to the 2D example of \cref{sec:stseeg}. 
The \(Y\)-evolution~\cref{eq:stsetoyXYY} has \(G=g_{11}(t,X,\sigma) =-4 \sigma w(t) -2 X^{2} \).
For this theory we account for the  \(\sigma\)-dependence by defining \(\Xv=(X,\sigma)\), and then 
\(\grad g_{11}=-4(X,w)\), which gives  
\(g_{11}\grad g_{11}=8(2\sigma w+X^2)(X,w)\).
Consequently, for \(\sigma\geq0\) and when \(|w(t)|\leq1\) as in \cref{sec:stseeg}, \(|g_{11}\grad g_{11}|\leq8|2\sigma+X^2|\sqrt{1+X^2}\).
Thus, in \(|X|<r\) this bound gives \(|g_{11}\grad g_{11}| <G'_{\max}\leq 8(2\sigma+r^2)\sqrt{1+r^2} \).
Since here \(\beta=\cond Q=1\) and \(\delta=0\)\,, the radius~\(r\) of the ball~\(B_\mu\) in the lemma requires \(r<(1-\mu)^2/\big[ 16 (2\sigma+r^2) \sqrt{1+r^2} \big]\); that is, \(\mu<1-\sqrt{16 (2\sigma+r^2) \sqrt{1+r^2}}\). 
This \cref{lem:nfftd} certifies that the domain of validity of the conjugate system~\eqref{eqs:stsetoyxy}+\cref{eqs:stsetoyXY} is large enough to contain the balls of radius plotted in \cref{fig:rofmu}.
\cref{lem:exptri} guarantees, for every \(0<\mu<1\) as plotted in \cref{fig:rofmu}, that solutions decay to \(Y=0\) at least as fast as~\(e^{-\mu t}\) while inside~\(B_\mu\).
\begin{figure}
\centering
\begin{tabular}{@{}ll@{}}
    \parbox[t]{0.45\textwidth}{%
    \caption{\label{fig:rofmu}%
    lower bound of the radius~$r$ of the ball~\(B_\mu\) for the example of \cref{sec:stseeg}.  Two cases are plotted: \(\sigma=0\) has no time-dependent forcing; \(\sigma=1/4\) is for every integrable forcing satisfying \(|w(t)|\leq1\), such as \(w=\cos t\) for the simulations in \cref{fig:stsesim1}.
    }}
&
\raisebox{-\height}{\input{stseeg.ltx}}
\end{tabular}
\end{figure}%

Extant forward theory also indicates such rates of decay in some small enough neighbourhood.  
But it is this backwards approach that gives a novel lower bound for the size of the neighbourhood.
\end{example}

\begin{proof}[Proof of \cref{lem:nfftd}]
By \cref{def:nftpd}, the size of~\(\DD_\mu\) and hence~$\MM_{cs}$ is limited by the condition \(\mu+\delta<\beta-\cond(Q)\|{G}(t,\Uv)\|\); that is, by \(\|G\|^2<(\beta-\mu-\delta)^2/\cond^2 Q\)\,.
Since \(\|G\|^2\leq\|G\|_F^2:=\sum_{i,j}g_{ij}^2\) (from Mirsky's theorem), this condition is assured by \(\|G\|_F^2<(\beta-\mu-\delta)^2/\cond^2 Q\)\,.
By a multivariable remainder theorem in~\(\Uv\), \(\|G\|_F^2=\|G(t,\ov)\|_F^2+\Uv\cdot\grad(\|G\|_F^2)\big|_{\xi\Uv}\) for some \(0\leq\xi\leq1\)\,.
Since \(G(t,\ov)=0\), this gives 
\begin{equation*}
\|G\|^2\leq\|G\|_F^2=\Uv\cdot\grad(\|G\|_F^2)\big|_{\xi\Uv}
\leq|\Uv|\left|\sum_{i,j}2g_{ij}\grad g_{ij}\right|_{\xi\Uv}
\leq2|\Uv|G'_{\max}\,.
\end{equation*}
Consequently the requisite condition holds when \(2|\Uv|G'_{\max}<(\beta-\mu-\delta)^2/\cond^2 Q\) which rearranges to the criterion in \cref{lem:nfftd}.
\end{proof}

The next step is to map the invariant manifolds of the normal form system~\cref{eqs:stsesdenf} into the original variables~\(\xv\), \(\yv\) and~\(\zv\).
Then we relate the system~\cref{eq:stsexform}+\cref{eqs:stsesdenf}, with its exact invariant manifolds, to the original system~\cref{eqs:stsesde}.
To understand the mapping, we characterise the near identity, polynomial, coordinate transform~\cref{eq:stsexform} constructed by the induction \text{of \cref{sec:stsenfsmsde}.}

Our approach is different to that used by \cite{Murdock03} [Chap.~5]  for autonomous systems.  
In essence Murdock uses the coordinate transform~\cref{eq:stsexform} to map the original dynamics~\cref{eqs:stsesde} into the \(\Xv\Yv\Zv\)-space, calls that the \emph{full system} and typically denotes it by~\(\dot x=a(x)\) with \(x\) denoting the transformed variables \cite[p.296]{Murdock03}.  
He calls the normal form system~\cref{eqs:stsesdenf} the \emph{truncated system}, denoted \(\dot y=\widehat a(y)\).  
Murdock then compares invariant manifolds of these two systems; that is, he compares the invariant manifolds in what we call the \(\Xv\Yv\Zv\)-space.  
In contrast, here we use the coordinate transformation~\cref{eq:stsexform} to map the invariant manifolds of the normal form~\cref{eqs:stsesdenf} back into the \(\xv\yv\zv\)~variables of the original `physical' system~\cref{eqs:stsesde} and aim to make comparisons in that original space of the `physical' variables.

\begin{full}
Recall that we mostly use capital letters for quantities in the $\Xv\Yv\Zv$-space, such as the domains~$D_\mu$, and mostly use lowercase letters for quantities in the $\xv\yv\zv$-space, such as corresponding domains~$d_\mu$.
\end{full}

\begin{definition}[diffeomorphic domain]  \label{def:nfdd}
For every given coordinate transformation~\cref{eq:stsexform} constructed to order~\(p\) for \cref{thm:stsenfp}, define a (star-shaped) domain $D_\mu\subseteq \DD_\mu$ such that the coordinate transform~\cref{eq:stsexform} is a $\cC^p$-diffeomorphism
onto a domain~$d_\mu\subseteq d$.
\end{definition}

Since the coordinate transform~\cref{eq:stsexform} is near identity, then by continuity of derivatives, this domain~$D_\mu$ is a (finite) neighbourhood of the equilibrium at the~origin.

\subsection{Invariant\slash integral manifolds exist}
\label{sec:nfsime}

Instead of proving there exists a center manifold for the specified system~\cref{eqs:stsesde}, which is then approximately constructed, the next \cref{thm:nfimex} establishes that there is a system, such as~\cref{eq:stsexform}+\cref{eqs:stsesdenf}, `close' to the specified system and that has invariant\slash integral manifolds which we know exactly.
This is an example of a `backward error' theory \cite[e.g.,][]{Grcar2011}, whereas all previous invariant\slash integral manifold theory addresses `forward errors'.

In applications we \emph{never} know the exact mathematical models: \emph{all} our mathematical models are approximate.
Consequently, a slight enough perturbation to a prescribed mathematical model may well be as good a description of reality as the prescribed model.  
The backward approach here establishes properties about such slightly perturbed models.
Thus the `backward' theorems proved here may be just as useful in applications as `forward' theorems. 
Moreover, our backward theorems should apply to a significantly wide and  useful class of dynamical systems.  

The following \cref{thm:nfimex} establishes the domain of existence of manifolds invoked by \cref{cor:stbackex,pro:stbackem}. 
The theorem establishes `asymptotic closeness': future research is planned to derive a bound on the `closeness'.

\begin{theorem}[invariant\slash integral manifolds exist] 
\label{thm:nfimex}
Consider any dynamical system~\cref{eqs:stsesde} satisfying \cref{ass:givensys}. 
For all orders \(2\leq p\leq\fp\) and a chosen rate~\(\mu\), there 
exists a dynamical system which is both \Orc{p}~close to the 
system~\cref{eqs:stsesde}, and which possesses 
center, stable, unstable, center-stable, center-unstable manifolds (\cref{def:nfnfim}), denoted respectively by~\(\cM_c\), \(\cM_s\), \(\cM_u\), \(\cM_{cs}\) and~\(\cM_{cu}\), in a domain~\(d_\mu\) (\cref{def:nfdd}) for time interval~\(\TT_\mu\).%
\footnote{Although this \cref{thm:nfimex} only asserts that there is a dynamical system, there are vastly many such dynamical systems.
Firstly, there is the freedom identified at lower orders in the inductive proof of \cref{thm:stsenfp}.
Secondly, there is considerable freedom in choosing higher order terms in the polynomials.}
\end{theorem}

\begin{proof} 
The hard work has already been done.
For every order~\(p\), \cref{thm:stsenfp} establishes that there exists such a dynamical system that is \Orc{p}~close: namely the combination of the normal form~\cref{eqs:stsesdenf} together with the coordinate transform~\cref{eq:stsexform} define sufficiently close dynamics in the state space of~\cref{eqs:stsesde}.

\cref{def:nfnfim} establishes the existence of the requisite invariant\slash integral manifolds for the normal form~\cref{eqs:stsesdenf} in~\(D_\mu\).
The coordinate transform~\cref{eq:stsexform} maps these into corresponding invariant\slash integral manifolds in~\(\TT_\mu\otimes d_\mu\) for the corresponding system~\cref{eq:stsexform}+\cref{eqs:stsesdenf} in the original state space.
\end{proof}

Some researchers explore the possibility of exponentially small errors in asymptotic statements \cite[e.g.,][]{Jones96b, Cotter06, Iooss2010}.
This possibility arises here immediately from \cref{thm:nfimex} under two further restrictions.
In the case when the eigenvalues of the center modes have precisely zero real-part, \(\alpha=0\)\,, and when the specified dynamical system~\cref{eqs:stsesde} is infinitely continuously differentiable, \(\fp=\infty\)\,, then the iterative construction of the normal form that proves \cref{thm:stsenfp} may be continued to arbitrarily high order~\(p\), in principle:
the reason being that the two constraints on the order~\(p\) are that firstly \(p<\tfrac12(\beta/\alpha+1)=\infty\) and secondly \(p\leq\fp=\infty\) (\cref{ass:givensys}).
To achieve the exponential closeness, for brevity, let \(\epsilon=|(\Xv,\Yv,\Zv)|\)\,, and for any given~\(\epsilon\) choose order~\(p=-c/(\epsilon\log\epsilon)\) for some constant~\(c\).
Then in \cref{thm:nfimex} the order of closeness
\begin{equation*}
\Orc p=\Ord{|\Xv|^p+|\Yv|^p+|\Zv|^p}
=\Ord{\epsilon^p}
=\Ord{e^{p\log\epsilon}}
=\Ord{e^{-c/\epsilon}}
\quad\text{as }\epsilon\to0\,.
\end{equation*}
This exponentially small closeness is a little subtle as it requires higher and higher order construction as \(\epsilon\to0\)\,.
Nonetheless the argument of this paragraph establishes the following corollary of \cref{thm:nfimex}.

\begin{corollary}[exponentially small closeness] \label{cor:nfesc}
Suppose the dynamical system~\cref{eqs:stsesde} satisfies \cref{ass:givensys} for the case of \(\alpha=0\) and infinitely differentiability, \(\fp=\infty\). 
Then there exists a dynamical system exponentially close to the system~\cref{eqs:stsesde}, the difference is~\Ord{\exp\big[-c/|(\Xv,\Yv,\Zv)|\big]} as $(\Xv,\Yv,\Zv)\to\ov $ for some~\(c\),  with center, stable, unstable, center-stable, center-unstable manifolds in~\(\TT_\mu\otimes d_\mu\) (provided the domain does not degenerate as \(\epsilon\to0\)).
\end{corollary}

An immediate partnering consequence of the existence \cref{thm:nfimex} addresses the evolution of the approximating system~\cref{eq:stsexform}+\cref{eqs:stsesdenf} on its invariant\slash integral manifolds in the original state space.
The following \cref{cor:nfeom} provides the evolution invoked in \cref{cor:stbackex,pro:stbackem}.

\begin{corollary}[evolution on manifolds] \label{cor:nfeom}
The evolution of the approximate system~\cref{eq:stsexform}+\cref{eqs:stsesdenf} in~\(\TT_\mu\otimes d_\mu\) on any of the invariant\slash integral manifolds~\(\cM_i\), \(i\in\{c,s,u,cs,cu,0\}\), is described by the system~\cref{eqs:stsesdenf} restricted to~\(\MM_i\) and transformed by~\cref{eq:stsexform}.
\end{corollary}

\begin{full}
Centre manifolds are crucial to accurate model reduction of dynamical systems \cite[e.g.,][]{Roberts2014a}.
\end{full}%
Because of the form of system~\cref{eqs:stsesdenf}, the evolution on the center manifold~\(\MM_c\) (\(\Yv=\ov \) and \(\Zv=\ov \)), is \(\dot{\Xv}=A\Xv+\Fv_c(t,\Xv)\).
The coordinate transform~\cref{eq:stsexform} maps this evolution into the original state space to give, as invoked in~\cref{eqs:stbackex}, a parametric description of~\(\cM_c\) and the evolution thereon as
\begin{equation}
\begin{bmatrix} \xv\\
\yv\\
\zv \end{bmatrix}
=\begin{bmatrix} \xv(t,\Xv,\ov ,\ov )\\ 
\yv(t,\Xv,\ov ,\ov )\\
\zv(t,\Xv,\ov ,\ov ) \end{bmatrix}
\quad\text{such that}\quad
\dot{\Xv}=A\Xv+\Fv_c(t,\Xv).
\label{eq:nfcmeX}
\end{equation}

\begin{full}
Most people simplify the parametrisation of the center manifold~\(\cM_c\) by choosing, often implicitly, that the new coordinate~\(\Xv\) be equal to the original~\(\xv\) on~\(\cM_c\).
\cref{egUdwadia2022,eg:qga} make this choice for their systems, but \cref{sec:stseeg} does not.
It is a subjective choice, and, if desired, one may be flexible about how to parametrise the center manifold (\cref{sec:stsenfsmsde}).
Nonetheless, in the many cases when people choose the coordinate transform so that on~\(\cM_c\) \(\xv(t,\Xv,\ov ,\ov )=\Xv\) precisely
\cite[e.g.,][]{Chicone97}, then the evolution on the center manifold~\(\cM_c\) is that on the graph
\begin{equation}
\yv=\yv(t,\xv,\ov ,\ov ),\quad
\zv=\zv(t,\xv,\ov ,\ov ) 
\quad\text{such that}\quad
\dot{\xv}=A\xv+\Fv_c(t,\xv),
\label{eq:nfcmex}
\end{equation}
for  \((\xv,\ov ,\ov )\in d_\mu\) and \(t\in\TT_\mu\).
\end{full}

\subsection{The center manifold dynamics emerge}

Centre manifolds provide exceptionally powerful theory and techniques for modelling emergent dynamics in complex systems \cite[e.g.,][]{Potzsche08, Roberts2014a}.
This section establishes the crucial \cref{pro:stbackem} that for systems~\eqref{eq:stsexform}+\eqref{eqs:stsesdenf}, and for initial conditions in the centre-stable manifold, the solutions approach a solution on a {center manifold} exponentially quickly.
That is, the dynamics on the center manifold predict the dynamics of the full system apart from exponentially quickly decaying transients.

A departure from other extant theorems \cite[e.g.,][]{Carr81} is that we relax the `straightjacket' that solutions are required to remain in the neighbourhood of the reference equilibrium for all time (as also relaxed by \cite{Kobayasi03}).
The following lemma establishes that solutions are exponentially 
quickly attracted to the center manifold over a finite time---although preferable, all time is not required in this approach.
Thus even if some solutions exit the domain of validity of 
the center manifold model, we are empowered to use the center manifold model until they do so~exit.

\begin{full}
A corresponding such theorem should be especially useful in generalisations to stochastic dynamics. 
There the problem is that an inevitable rare event will eventually occur to push the stochastic system out of the domain of validity.
Similarly, \cite{Berglund03} do prove some theory up until the first exit time (see their Theorem~2.4).
Nonetheless, the almost certain eventual occurrence of such extremely rare events plagues and strongly constrains most established `forward' theorems on stochastic modelling.  
\end{full}

\begin{lemma}[emergent dynamics] \label{lem:nfemerge}
Consider the class of normal form systems~\cref{eqs:stsesdenf}, in a suitable domain~\(\DD_\mu\) (\cref{def:nftpd}).
For every initial condition $(\Xv_0,\Yv_0,\ov )\in \MM_{cs}$ and \((\Xv_0,\ov,\ov )\in \MM_{c}\) at \(t_0\in\TT_\mu\),\footnote{In the important case when there are no unstable modes---that is, when \(\Zv(t)\)~is absent---the center-stable manifold \(\MM_{cs}= \DD_\mu\).} solutions~$(\Xv(t),\Yv(t),\ov )$ of the normal form~\cref{eqs:stsesdenf} are exponentially quickly attracted to the solution~$(\Xv(t),\ov ,\ov )$ on the center manifold~\(\MM_c\) in the sense that $|(\Xv(t),\Yv(t))-(\Xv(t),\ov )|\leq c|\Yv_0|e^{-\mu t}$ for some constant~\(c\) and for all $t_0\leq t\leq T_\mu$ where the first exit time~$T_\mu$ is such that both $(\Xv(t),\Yv(t),\ov)\in \MM_{cs}$ and $(\Xv(t),\ov,\ov )\in \MM_{cs}$ for all $t\in [t_0,T_\mu)\subseteq\TT_\mu$.
\end{lemma}

\begin{example}[a cylinder of attraction] \label{eg:nfaca}
Consider the normal form autonomous system in center variables \(\Xv=(X_1,X_2)\) and stable variables \(\Yv=(Y_1,Y_2)\)
\begin{equation*}
\de t{\Xv}=\Fv(\Xv),\quad
\de t{\Yv}=\begin{bmatrix} -1&1\\-1&-1 \end{bmatrix}\Yv
+\begin{bmatrix} X_1^2Y_1-X_1X_2Y_2\\
-X_1X_2Y_1+X_2^2Y_2 \end{bmatrix}.
\end{equation*}
There are no unstable variables~\(\Zv\) in this example.
Here matrices determining the stable dynamics are
\begin{equation*}
B=\begin{bmatrix} -1&1\\-1&-1 \end{bmatrix}
\qtq{and}
G=\begin{bmatrix} X_1^2&-X_1X_2\\
-X_1X_2& X_2^2 \end{bmatrix}.
\end{equation*}
The linear matrix~\(B\) of the stable variables is already in real Jordan form, corresponding to eigenvalues \(-1\pm i\), so choose `diagonalising' matrix \(Q=I\) for which the condition number \(\cond Q=1\) (\cref{def:nftpd}).  
As the nonlinearity matrix~\(G\) is symmetric we find its \(2\)-norm from the largest eigenvalue: its two eigenvalues are zero and \(X_1^2+X_2^2\); the largest gives the norm \(\|G\|=X_1^2+X_2^2\).
The eigenvalues of~\(B\) are~\(\beta_j=-1\pm i\) so an upper bound on their real-part is \(-\beta=-1\)\,.
For every decay rate \(0<\mu<1=\beta\) the domain~\(\DD_\mu\) is then constrained by \(\mu<\beta-\cond Q\,\|G\|=1-(X_1^2+X_2^2)\).  
That is, domain~\(\DD_\mu\) is at least the cylinder \(X_1^2+X_2^2<1-\mu\) for all~\Yv.
\cref{lem:nfemerge} proves that while solutions stay within this specific cylindrical domain~\(\DD_\mu\), solutions decay to the center manifold \(\Yv=\ov \) through being bounded by \(|\Yv|\leq|\Yv_0|e^{-\mu t}\).
\begin{full}
Again, extant forward theory also gives such rates of decay in some neighbourhood.  But it is this backwards approach that gives a lower bound for the size of the neighbourhood.
\end{full}%
\end{example}

The following proof of \cref{lem:nfemerge} could be extended to establish some conditions for when the domain~\(\DD_\mu\) of emergence is much larger than that guaranteed by \cref{def:nftpd}.

\begin{proof}[Proof of \cref{lem:nfemerge}]
Given any solution~$(\Xv(t),\Yv(t),\ov)$ of~\cref{eqs:stsesdenf}, because the $\Xv$-equation~\cref{eq:stsesdenfxx} is independent of~$\Yv$ when \(\Zv=\ov\), we have that $(\Xv(t),\ov,\ov )$~is also a solution of the normal form~\cref{eqs:stsesdenf}.  
Further, let \(T_\mu=\sup\{T\in\TT_\mu: \forall t\in[t_0,T),\ (\Xv(t),\Yv(t),\ov),(\Xv(t),\ov,\ov)\in\DD_\mu\}\).
Then by the exponential trichotomy \cref{lem:exptri} the distance between them $|(\Xv,\Yv,\ov)-(\Xv,\ov,\ov )|=|\Yv|\leq c|\Yv_0|e^{-\mu (t-t_0)}$ for all $t_0\leq t< T_\mu$ and constant \(c:=\cond Q\).
\end{proof}

The previous \cref{lem:nfemerge} establishes a finite size 
domain in which a center manifold model emerges in time.
However, it only applies to systems in the special normal form~\cref{eqs:stsesdenf}.
\cref{pro:stbackem} uses the polynomial diffeomorphism of \cref{def:nfdd} to prove similar emergence in a wide class of dynamical systems.
\cref{thm:nfimex} establishes that there is a member of this wide class close to any specified dynamical system in an even wider and useful class.

\begin{proof}[Proof of \cref{pro:stbackem}]
To realise the exponential decay of distance between general solutions of the system~\cref{eqs:stsesde} and the center manifold solution, consider the trajectory starting from $(\xv_0,\yv_0,\zv_0)\in \cM_{cs}\subseteq d_\mu$ at time \(t_0\in\TT_\mu\): 
\begin{itemize}
\item it maps to the trajectory of~\cref{eqs:stsesdenf} starting from some point~$(\Xv_0,\Yv_0,\ov )\in D_\mu\subseteq \DD_\mu$; 
\item by \cref{lem:nfemerge}, this trajectory approaches exponentially quickly to the solution starting from~$(\Xv_0,\ov ,\ov )$; 
\item hence starting the model $\dot{\vec s}=A\vec s+\Fv_c(t,\vec s)$ with initial condition $\vec s(0)=\Xv_0$ gives the requisite solution on the center manifold~\(\cM_c\) approached by the trajectory from the specified initial condition.
\end{itemize}
The constant \begin{equation}
C:=\cond Q\,|\Yv_0|\Lip
\label{eq:bigC}
\end{equation}
where \(Q\)~is the similarity matrix introduced in the proof of \cref{lem:nfemerge}, and where $\Lip$ is a Lipschitz constant of the coordinate transform~\cref{eq:stsexform} on~\(\MM_{cs}\): $|(\xv(t,\Xv,\Yv,\ov ),\yv(t,\Xv,\Yv,\ov ))-(\xv(t,\Xv,\ov ,\ov ),\yv(t,\Xv,\ov ,\ov ))|\leq\Lip|\Yv|$ for $(\Xv,\Yv,\ov )\in \MM_{cs}\subseteq \DD_\mu$ and \(t\in\TT_\mu\)\,.
\end{proof}

\begin{full}
\begin{remark}
If one insists, in the coordinate transform~\cref{eq:stsexform}, that $\xv=\Xv$ when on the center manifold $\Yv=\ov $\,, then the center manifold may be more simply expressed as $\yv=\yv(\xv,\ov )$.
For example, to force $\xv=\Xv$ in the construction of a polynomial coordinate transform,  in~\eqref{eq:stsemux} set $\fx=0$ and always update the evolution with non-zero~\(\ff\).
\end{remark}
\end{full}

\begin{corollary}[center-unstable dynamics] \label{cor:cud}
The previous \cref{lem:nfemerge,pro:stbackem} immediately also apply to the center-unstable manifold of system~\cref{eqs:stsesde} when considered backward in time.
\end{corollary}

\subsection{Singular perturbation dynamics emerge instantaneously}
\label{sec:nfssppei}

Many researchers choose to phrase problems as singular perturbations \cite[e.g.,][]{Bykov2013, Pavliotis07, Verhulst05}.
For problems phrased as singular perturbations, this section establishes a new backwards view of how slow manifolds exist and `instantly' emerge over large finite domains.

Let's consider the class of autonomous singular perturbation dynamics governed by
\begin{equation}
\dot{\uv }=\cF(\uv ,\vec v) \qtq{and}
\dot{\vec v}=\frac1\epsilon\cG(\uv ,\vec v),
\label{eq:nfspp}
\end{equation}
for \(\uv (t)\in\RR^m\), \(\vec v(t)\in\RR^n\) and the regime where parameter~\(\epsilon\) is small.
The heuristic singular perturbation argument is that as parameter \(\epsilon\to0\) the \(\vec v(t)\)~dynamics are very fast and will rapidly settle onto an `equilibrium' of the \(\vec v\)-\ode.%
\footnote{We restrict attention to this scenario of rapid attraction of fast variables~\(\vec v\) to an equilibrium.
In other singular perturbation scenarios the fast variable~\(\vec v\) is rapidly attracted to an invariant distribution reflecting either rapid oscillations, chaos, or stochasticity in the \(\vec v\)-\ode\ \cite[e.g.,][]{Berglund03}.} 
Hence solving \(\cG(\uv ,\vec v)=\ov \) gives quasi-equilibria \(\vec v=\cV(\uv )\) parametrised by the `frozen' slow variable~\(\uv \).
Then the argument is that the slow variables are not truly frozen but instead evolve according to the \(\uv \)-\ode, namely \(\dot{\uv }\approx\cF(\uv ,\cV(\uv ))\).
Indeed some beautiful theorems \cite[e.g.,][]{Pavliotis07, Verhulst05} establish the slow manifold model that
\(\vec v=\cV(\uv )+\Ord\epsilon\) {such that}
\(\dot{\uv }=\cF(\uv ,\cV(\uv ))+\Ord\epsilon\).
Let's view this scenario using our normal form coordinate transformations.
In particular, and in contrast to the common singular perturbation theory, we do not use the limit `\(\epsilon\to0\)' but treat parameter~\(\epsilon\) as a fixed finite value, albeit notionally small in effect.
In applications, \(\epsilon\)~is almost always \text{finite valued.}

To establish our coordinate transform view, let's choose to embed the original singular~\cref{eq:nfspp} as the \(\theta=1\) member of the family of systems
\begin{equation}
\dot{\uv }=\theta\cF(\uv ,\vec v) \qtq{and}
\dot{\vec v}=\frac1\epsilon\cG(\uv ,\vec v),
\label{eq:nfsppp}
\end{equation}
for homotopy parameter~\(\theta\), for at least \(0\leq\theta\leq1\)\,.
Then we analyse this family, and set parameter \(\theta=1\) to recover results about the original~\cref{eq:nfspp}.
One might imagine parameter~\(\theta\) is a `temperature' in that when \(\theta=0\) the slow variables~\(\uv \) are `frozen', but when \(\theta=1\) the system has `warmed' to become the original `out-of-equilibrium', finite~\(\epsilon\), system.

Now proceed along familiar lines.
First find equilibria: 
the system~\cref{eq:nfsppp} has a manifold of equilibria for parameter \(\theta=0\) and \(\vec v^*=\cV(\uv ^*)\).
Second, change to coordinates local to each equilibria.
We introduce new slow and fast variables, familiarly called \(\xv\) and~\(\yv\), according to the linear transformation
\begin{equation*}
\begin{bmatrix} \uv \\\vec v \end{bmatrix}
=\begin{bmatrix} \uv ^*\\\cV(\uv ^*) \end{bmatrix}
+\begin{bmatrix} I&0\\ L&I \end{bmatrix}
\begin{bmatrix} \xv\\\yv \end{bmatrix}
\end{equation*}
for \(\xv(t)\in\RR^m\) and \(\yv(t)\in\RR^n\). 
The submatrix \(L(\uv ^*):=-\left[\D{\vec v}\cG\right]^{-1}\D{\uv }\cG\) evaluated at \((\uv ^*,\cV(\uv ^*))\).
This choice for the submatrix~\(L(\uv ^*)\) ensures the system~\cref{eq:nfsppp} becomes linearly separated:
\begin{equation}
\dot{\xv}=\theta\fv(\xv,\yv) \qtq{and}
\dot{\yv}=\frac1\epsilon B\yv+\frac1\epsilon\gv(\xv,\yv),
\label{eq:nfspxy}
\end{equation}
where matrix \(B(\uv ^*):=\D{\vec v}\cG\) evaluated at \((\uv ^*,\cV(\uv ^*))\),  and the function \(\fv(\xv,\yv):=\cF(\uv ^*+\xv,\cV(\uv ^*)+L\xv+\yv)\) is also implicitly a function of~\(\uv ^*\), as is the function \(\gv(\xv,\yv):=\cG(\uv ^*+\xv,\cV(\uv ^*)+L\xv+\yv)-B\yv-\epsilon\theta\fv(\xv,\yv)\) which is also implicitly a function of the small product~\(\epsilon\theta\) (only the leading dependence in parameters~\(\theta\) and small~\(\epsilon\) is explicit).
The function~\(\gv=\Orc2\) by the choice of~\(L\) and~\(B\).
Further, functions \(\fv\) and~\(\gv\) are as smooth as \(\cF\) and~\(\cG\) in the corresponding domains. 
We also require that in the domain, \(\det(\D{\vec v}\cG)\) be bounded away from zero as is consistent with the singular perturbation assumption that the fast variables~\(\vec v\) evolve rapidly to an quasi-equilibrium. 
In the linearly separated form~\cref{eq:nfspxy} (with no unstable variables~\zv) we readily apply the results of the previous sections.

To quantify the separation of time scales, suppose all eigenvalues of~\(B\) have negative real-part bounded away from zero: \(\Re\beta_j\leq-\beta^*<0\) (\(\beta^*\)~depends upon~\(\uv ^*\)): thus the bound~\(\beta\) invoked in previous sections is here \(\beta^*/\epsilon\) (large since \(\epsilon\)~is small). 
Based about the equilibria \(\theta=0\), the linear matrix for the slow variables~\(\xv\) is zero, with eigenvalues that are zero so \(\alpha=0\).%
\footnote{Strictly, here \Orc{p} denote terms \Ord{\theta^{p/2}+|\Xv|^p+|\Yv|^p+|\Zv|^p} as $(\theta,\Xv,\Yv,\Zv)\to\ov $\,.}
\cref{sec:stsenfsmsde}, via \cref{thm:stsenfp}, establishes the existence of coordinate transforms which together with the normal form
\begin{equation}
\dot{\Xv}=\theta\Fv(\Xv) \qtq{and}
\dot{\Yv}=\frac1\epsilon B\Yv+\frac1\epsilon G(\Xv,\Yv)\Yv,
\label{eq:nfspXY}
\end{equation}
gives a system asymptotically close to the original~\cref{eq:nfspxy} to any specified order (an order limited only by the smoothness of~\cF\ and~\cG\ in the chosen domain).
By rescaling time with~\(\epsilon\), one can see that the coordinate transform and the normal form~\cref{eq:nfspXY}, apart from the explicit factors shown above, depend upon parameters \(\epsilon\) and~\(\theta\) only via the small product~\(\epsilon\theta\).

Now determine the domain of emergence from the results of \cref{sec:nfeep}.
We need to choose a rate parameter \(\mu^*<\beta^*/\epsilon\), say choose \(\mu^*:=\beta^*/\sqrt\epsilon\).
Then for the normal forms~\cref{eq:nfspXY} the star-shaped domain~\(\DD_\mu^*\) must satisfy \(\mu^*<\beta^*/\epsilon-\cond Q^*\,\|G\|/\epsilon\); that is, \((\Xv,\Yv)\in\DD_\mu^*\) must satisfy
\begin{equation*}
  \|G(\Xv,\Yv)\|<\frac{\beta^*(1-\sqrt\epsilon)}{\cond Q^*} 
  \,.
\end{equation*}
In contrast, extant forward theory does not provide such a bound on the domain size.%
\footnote{If the limit \(\epsilon\to 0\) is taken, then this bound \({\beta^*(1-\sqrt\epsilon)}/{\cond Q^*} \to {\beta^*}/{\cond Q^*}\).}
The superscript~\(*\) on quantities indicates that they depend upon the location~\(\uv ^*\) of the base equilibria of the analysis; for example, \(\DD_\mu^*:=\DD_\mu(\uv ^*)\).
As \(G=\Orc1\), these local domains~\(\DD_\mu^*\) exist for every value of the singular perturbation parameter~\(\epsilon<1\).
Further, the domains exist at homotopy parameter~\(\theta=1\) for  small enough~\(\epsilon\) as the homotopy parameter only occurs in the combination~\(\epsilon\theta\).
Define the domain \(\DD_\mu:=\bigcup_{\uv ^*}\DD_\mu^*\), global over~\(\uv^*\), which also contains \(\theta=1\) for small enough~\(\epsilon\).
Since the attractiveness of the slow manifold is ensured inside each~\(\DD_\mu^*\), the slow manifold is attractive in the union~\(\DD_\mu\).
Letting \(\beta:=\min_{\uv ^*}\beta^*\) and \(\mu:=\beta/\sqrt\epsilon\), \cref{lem:nfemerge} asserts all solutions of the normal form~\cref{eq:nfspXY} in the union~\(\DD_\mu\) are attracted to a slow manifold solution at least as fast as \(\Ord{e^{-\mu t}}=\Ord{e^{-\beta t/\sqrt\epsilon}}\).
That is, if one invokes the limit as the singular perturbation parameter \(\epsilon\to 0\), then this attraction happens  `instantaneously' in time.

But how does this existence and attraction translate to dynamics~\cref{eq:nfspp} in the original slow\slash fast variables~\((\uv ,\vec v)\)?
We proceed via the linearly transformed dynamics~\cref{eq:nfspxy} of the local variables~\((\xv,\yv)\).
The coordinate transforms to/from variables~\((\Xv,\Yv)\) are near identity polynomial, and so are $\cC^p$-diffeomorphism in some domain \(D_\mu\subseteq \DD_\mu\).
As in \cref{def:nfdd}, let the domain \(d_\mu\subseteq\RR^{m+n}\) be the image of~\(D_\mu\) under the coordinate transform.
Since the coordinate transform is near identity, and depends upon \(\epsilon\) and~\(\theta\) only in the combination~\(\epsilon\theta\), the physically relevant case of parameter \(\theta=1\) lies in domain~\(d_\mu\) for small enough~\(\epsilon\).
The linear transformation then maps the domain~\(d_\mu\) into the original variables~\((\uv ,\vec v)\).
Thus we are assured that there is a domain, global across the set of equilibria~\((\uv ^*,\vec v^*)\) found at \(\epsilon=0\) (provided \(\D{\vec v}\cG\) has eigenvalues with real-part bounded away from zero), in which a slow manifold exists and in which all solutions are attracted exponentially quickly, at least as fast as~\Ord{e^{-\mu t/\sqrt\epsilon}}, to solutions \text{on the slow manifold.}

This coordinate transform view connects to the existence and rapid emergence of slow manifolds in for the many problems phrased as singular perturbation.

\section{Conclusion}

This article establishes a new foundation in a complementary theory of invariant manifolds for non-autonomous dynamical systems.
Results on the existence and emergence of center manifolds, \cref{cor:stbackex,pro:stbackem}, are based upon being theoretically able to construct (\cref{thm:stsenfp}) an approximate normal form~\cref{eqs:stsesdenf} corresponding to any prescribed system~\cref{eqs:stsesde}.
In the normal form system we readily identify invariant manifolds (\cref{def:nfnfim}) within a finite domain (\cref{def:nftpd,lem:nfftd}) with an associated exponential trichotomy (\cref{lem:exptri}) for at least a  finite-time.
Consequently, we deduce that the normal form dynamics in its center-stable manifold is exponentially quickly attracted to its center manifold (\cref{lem:nfemerge}).
By transforming back~\cref{eq:stsexform} from the normal form~\cref{eqs:stsesdenf} we establish the existence and emergence of invariant manifolds (\cref{cor:stbackex,pro:stbackem}) in a finite domain for many systems `arbitrarily close in an asymptotic sense' to the specified system~\cref{eqs:stsesde}. 
Future research is planned on the outstanding issue of how to quantify closeness in terms of norms rather than being simply asymptotic.
Another outstanding challenge is to establish such quantitative information direct from the algebraic form of the original `physical' system~\cref{eqs:stsesde}.

The non-autonomous theory developed here is based upon on the construction of a coordinate transform conjugacy (\cref{sec:stsenfsmsde}).
This construction requires the existence of the convolution integrals~\cref{eq:stsezmuf}.
This requirement is only a weak constraint on the non-autonomous nature of the dynamical system.
Consequently, the arguments developed here should also apply to many stochastic systems provided the sufficient ordinary rules of calculus hold such as in the Stratonovich interpretation \cite[e.g.,][]{vanKampen81}, or potentially the Marcus interpretation for jump processes \text{\cite[e.g.,][]{Chechkin2014}.}

The other innovation in this complementary approach is that the defining properties of the invariant\slash integral manifolds (\cref{def:nfnfim}) do not need as many restrictions on linear operators as that required by forward theory.
Consequently, future research should be able to develop the approach to establish center manifold theory for a wider range of partial differential systems than is currently available in applications \cite[e.g.,][]{Hochs2019}.
Although many technical details remain to be resolved including how to measure distances between \pde\ systems.

\paragraph{Acknowledgements}
I thank Profs.\ Georg Gottwald and C.~Chicone for interesting conversations on cognate issues, and thank reviewers for their constructive comments.
This project was partly supported by the Australian Research Council through grants DP120104260, DP150102385, and DP200103097.

\end{document}

%% file: ctECT.ltx
\tikzsetnextfilename{ctECT}
\begin{tikzpicture}
\begin{axis}[xlabel={$x$},ylabel={$y$},small
    ,ymin=-0.5,ymax=2,xmin=-1.5,xmax=1.5,axis equal
    ,label shift={-1ex},max space between ticks={100}
    ,cycle list={magenta,red,orange,yellow!90!black,black,green,cyan,blue,violet}
    ,restrict x to domain*=-4:4,restrict y to domain*=-4:4
    ]
\addplot[fill=blue,forget plot,domain=-1.5811:1.5811,opacity=0,fill opacity=0.2] {x^2-0.5};

    \foreach \X in {-0.8,-0.6,...,0.8} {
    \addplot+[no marks,domain=-0.99:0.45] (
    { (\X)/sqrt(1-2*\x/(1+2*(\X)^2))},
    {(\X)^2+\x/(1-2*\x/(1+2*(\X)^2))} );
    };
    \foreach \Y in {-0.8,-0.6,...,0.4} {
    \addplot+[smooth,no marks,domain=-0.9:0.9] (
    { \x/sqrt(1-2*(\Y)/(1+2*\x^2))},
    {\x^2+\Y/(1-2*(\Y)/(1+2*\x^2))});
    };  
\end{axis}
\end{tikzpicture}

%% file: ctECTx.ltx
\tikzsetnextfilename{ctECTx}
\begin{tikzpicture}
\begin{axis}[xlabel={$X$},ylabel={$Y$},small
    ,ymin=-0.5,ymax=0.5,axis equal
    ,label shift={-1ex},max space between ticks={100}
    ,unbounded coords=jump,samples=51
    ,cycle list={magenta,red,orange,yellow!90!black,black,green,cyan,blue,violet}
    ,restrict x to domain*=-4:4,restrict y to domain*=-4:4
    ]

\addplot[fill=blue,forget plot,domain=-1.1:1.1,opacity=0,fill opacity=0.2] {-0.9} \closedcycle;
\addplot[fill=blue,forget plot,domain=-1.1:1.1,opacity=0,fill opacity=0.2] {x^2+0.5} \closedcycle;

    \foreach \X in {0.8,0.6,...,-0.8} {
    \addplot+[no marks,domain=-0.45:1.9] (
    { (\X)/sqrt(1+2*(\x)-2*(\X)^2)},
    {((\x)-(\X)^2)*(1+2*(\x))/(1+2*(\x)-2*(\X)^2)^2} );
    };
    \foreach \Y in {0.8,0.6,...,-0.4} {
    \addplot+[no marks,domain=-0.9:0.9] (
    { (\x)/sqrt(1+2*(\Y)-2*(\x)^2)},
    {((\Y)-(\x)^2)*(1+2*(\Y))/(1+2*(\Y)-2*(\x)^2)^2} );
    };  
\end{axis}
\end{tikzpicture}

%% file: sim1.ltx
%
\tikzsetnextfilename{sim1}
\definecolor{mycolor1}{rgb}{0.13359,0.48867,0.73828}%
\definecolor{mycolor2}{rgb}{0.46758,0.53789,0.00000}%
\definecolor{mycolor3}{rgb}{0.77344,0.17578,0.16523}%
\definecolor{mycolor4}{rgb}{0.14766,0.56602,0.53438}%
\definecolor{mycolor5}{rgb}{0.74180,0.18984,0.45703}%
\definecolor{mycolor6}{rgb}{0.63633,0.48164,0.00000}%
\definecolor{mycolor7}{rgb}{0.37969,0.39727,0.68906}%
\begin{tikzpicture}

\begin{axis}[%
xmin=-0.615544512273764,
xmax=0.615544512273764,
xlabel style={font=\color{white!15!black}},
xlabel={$x$},
ymin=-0.45,
ymax=0.45,
ylabel style={font=\color{white!15!black}},
ylabel={$y$},
axis background/.style={fill=white},
axis equal
]
\addplot [color=mycolor1, forget plot]
  table[row sep=crcr]{%
-0.5	0.45\\
-0.455	0.379\\
-0.420511	0.336151928892062\\
-0.392239883245934	0.305141245318852\\
-0.368302169958457	0.278905730364427\\
-0.347757852817044	0.253973854866075\\
-0.330093572329065	0.228580216443238\\
-0.315003000287165	0.201874848037322\\
-0.302284763724311	0.173542231904816\\
-0.291792929210803	0.143602262582959\\
-0.283412504242727	0.112301684111619\\
-0.277046963937777	0.0800538676724423\\
-0.272611227739752	0.0474055936288878\\
-0.270026568323571	0.0150192493010801\\
-0.269215449054058	-0.0163363998472897\\
-0.270095051298221	-0.0457915964933772\\
-0.272568668019002	-0.0723813828013312\\
-0.276514447438909	-0.0950567351041652\\
-0.281771359555444	-0.112707562939282\\
-0.28812291220376	-0.124206149252473\\
-0.295280239691008	-0.128481209495195\\
-0.302867832158114	-0.12463203312628\\
-0.310417238896197	-0.112086140289304\\
-0.317375932933624	-0.0907891044275579\\
-0.323138788277205	-0.0613904779208045\\
-0.327106317206623	-0.0253612141785572\\
-0.32876547988059	0.0150365700077673\\
-0.327776778849718	0.0569819983484162\\
-0.324041303675505	0.0975089865721725\\
-0.31772191584972	0.133982835495862\\
-0.309208059212775	0.164471127322994\\
-0.299036899597559	0.187887040278181\\
-0.287799867997689	0.203900734576092\\
-0.276063347098563	0.212715382578413\\
-0.264318762999774	0.214826996388717\\
-0.252962201810889	0.210844097972252\\
-0.242295084358511	0.201386272744076\\
-0.232536103569876	0.187045393162833\\
-0.22383714218652	0.168383897160302\\
-0.216299028120402	0.145949468372644\\
-0.209985282487664	0.120293900712111\\
-0.204933292743148	0.0919906336321284\\
-0.201162904052796	0.0616493840810046\\
-0.198682590225836	0.0299281187293336\\
-0.19749335099589	-0.00245682996703527\\
-0.197590392512493	-0.0347218043445885\\
-0.198962531502331	-0.0660078054060341\\
-0.201589147514831	-0.0953740197564805\\
-0.205434420982385	-0.121794712485597\\
-0.210438586230024	-0.14416308330909\\
-0.216506081317649	-0.161308098373285\\
-0.22349091817037	-0.172033199543268\\
-0.231180489714712	-0.175188372833085\\
-0.239280516479487	-0.169787413763244\\
-0.247405880490884	-0.155176913784813\\
-0.255084216688242	-0.131248042239035\\
-0.261780077497524	-0.0986539768077664\\
-0.266945206636359	-0.0589602193163486\\
-0.270093036222105	-0.0146322064511911\\
-0.270883447635511	0.0312131253843286\\
-0.269192423832394	0.0752779890190817\\
};
\addplot [color=mycolor2, forget plot]
  table[row sep=crcr]{%
-0.3	0.45\\
-0.273	0.347\\
-0.2540538	0.293345528892062\\
-0.239148690734392	0.259217499744296\\
-0.226750385598534	0.233201714927442\\
-0.216174669842127	0.209926638945089\\
-0.207098505469122	0.186675046530236\\
-0.199366480840164	0.162096853942705\\
-0.192903144975	0.135615083012445\\
-0.187671029771171	0.107117834666952\\
-0.18365044690341	0.076788553880487\\
-0.180829996455966	0.0450104058037743\\
-0.179202150151571	0.0123127916516569\\
-0.178760854403902	-0.0206574122672503\\
-0.179499401737235	-0.0531499703414367\\
-0.181407479312963	-0.0843170540006954\\
-0.184466628158835	-0.113215279558584\\
-0.188643516334082	-0.138808454865959\\
-0.193880579338645	-0.159976513127672\\
-0.200083847147799	-0.175538189410903\\
-0.207108318399537	-0.184297690323214\\
-0.214742235345089	-0.185127817655083\\
-0.222693187622651	-0.177101373210794\\
-0.230581041489183	-0.159675249462108\\
-0.237944658553387	-0.132912776693906\\
-0.24426983560695	-0.097698062559729\\
-0.249042773543066	-0.0558597548283025\\
-0.251825067197441	-0.0101056345706097\\
-0.252334037618405	0.0362924594944133\\
-0.250502473050539	0.080019897782743\\
-0.246493436592974	0.118280918257683\\
-0.240662342588032	0.149188901545795\\
-0.233481512481202	0.171858967378013\\
-0.225456334153827	0.186234941307127\\
-0.217058764722136	0.192792313020177\\
-0.208689312459719	0.192259076018305\\
-0.200664829582039	0.185427198434226\\
-0.193223086147303	0.173059300872565\\
-0.186535275707085	0.155862030633429\\
-0.18072052233579	0.134494509415209\\
-0.175859338737228	0.109589850773829\\
-0.172004858999348	0.0817782061757992\\
-0.169191609234851	0.0517068861931922\\
-0.167441934978141	0.020056795543897\\
-0.166770265247075	-0.0124441156320343\\
-0.167185326940019	-0.0450094213401396\\
-0.168690309904445	-0.0767842026612872\\
-0.171280860092985	-0.106836584579438\\
-0.174940672512217	-0.134152115480312\\
-0.17963440477243	-0.157633953429288\\
-0.185297701051669	-0.176114157782578\\
-0.191824410763621	-0.188384333705465\\
-0.199051753525651	-0.19325689152703\\
-0.206745378153525	-0.189669696425735\\
-0.214588044775886	-0.176843773898502\\
-0.22217775671022	-0.154491396796087\\
-0.229042667104456	-0.123046278095712\\
-0.23467923664692	-0.0838508252360507\\
-0.238614846178642	-0.0392036324059163\\
-0.240485759921878	0.0078242146095097\\
-0.240109437482646	0.0538231884484879\\
};
\addplot [color=mycolor3, forget plot]
  table[row sep=crcr]{%
-0.1	0.45\\
-0.091	0.331\\
-0.0849758	0.271635128892062\\
-0.0803593175228588	0.235291072831639\\
-0.0765777515164644	0.208646407411638\\
-0.0733822169681501	0.185511962471819\\
-0.0706595611520913	0.162735799936236\\
-0.0683597931106455	0.138711906156909\\
-0.0664633296692716	0.112706097169492\\
-0.064965165170889	0.0845073107230028\\
-0.0638671588910355	0.0542332441546304\\
-0.0631744142467155	0.0222185583846954\\
-0.0628936853644435	-0.0110494735666055\\
-0.0630326737872317	-0.044966977845042\\
-0.0635995515563725	-0.078832207990134\\
-0.0646022901716468	-0.111853709640546\\
-0.0660474913330415	-0.143152402307515\\
-0.0679384627431835	-0.171761250528037\\
-0.0702723018071278	-0.196626553980603\\
-0.0730357819160519	-0.216616325413051\\
-0.0761999304565167	-0.230543653812325\\
-0.0797134125340548	-0.237215968737979\\
-0.08349527140919	-0.235523536759573\\
-0.0874282917341886	-0.224579715384135\\
-0.0913552159090256	-0.203917076839757\\
-0.0950809936254723	-0.173722426895825\\
-0.098384533818329	-0.135058545816367\\
-0.101042068231994	-0.0899814241083577\\
-0.102860450070868	-0.0414472582528619\\
-0.103713106798486	0.00704939167430758\\
-0.10356688393617	0.0520468933216796\\
-0.102488817024193	0.0907076972377363\\
-0.100629512107216	0.121202899629085\\
-0.0981901943760852	0.142770768231165\\
-0.0953864564793169	0.155503110132655\\
-0.0924198783499036	0.16001959093518\\
-0.0894620801243374	0.157176564454926\\
-0.0866498116437505	0.147877721081047\\
-0.0840870963081545	0.132984053058287\\
-0.0818506477327624	0.113290459181726\\
-0.0799960682395691	0.0895361528237253\\
-0.0785635602013301	0.062427105690318\\
-0.0775826610661108	0.0326595306914162\\
-0.0770758984060688	0.000940347728925213\\
-0.0770614027768646	-0.0319959390613806\\
-0.0775545331663112	-0.0653732060391674\\
-0.0785685308615017	-0.0983615992297748\\
-0.080114156130436	-0.130066859326897\\
-0.0821981954655381	-0.159521439118844\\
-0.084820670352265	-0.185679071232642\\
-0.0879705550107327	-0.207416348365077\\
-0.0916198612675278	-0.223547508063746\\
-0.0957161396026282	-0.232861844601225\\
-0.100173866967823	-0.234196228144689\\
-0.10486593532833	-0.226556080640565\\
-0.109617538388469	-0.209293050026847\\
-0.11420597617762	-0.182331386493422\\
-0.11837064297408	-0.146404158805109\\
-0.121836633856448	-0.103220132989246\\
-0.124351832566373	-0.0554545501748309\\
-0.12573100755405	-0.00647942339183702\\
};
\addplot [color=mycolor4, forget plot]
  table[row sep=crcr]{%
0.1	0.45\\
0.0910000000000001	0.331\\
0.0849758000000001	0.271635128892062\\
0.0803593175228589	0.235291072831639\\
0.0765777515164645	0.208646407411638\\
0.0733822169681502	0.185511962471819\\
0.0706595611520913	0.162735799936236\\
0.0683597931106456	0.138711906156909\\
0.0664633296692717	0.112706097169492\\
0.0649651651708891	0.0845073107230028\\
0.0638671588910356	0.0542332441546304\\
0.0631744142467155	0.0222185583846955\\
0.0628936853644436	-0.0110494735666055\\
0.0630326737872318	-0.0449669778450419\\
0.0635995515563726	-0.078832207990134\\
0.0646022901716469	-0.111853709640546\\
0.0660474913330416	-0.143152402307515\\
0.0679384627431835	-0.171761250528037\\
0.0702723018071279	-0.196626553980603\\
0.073035781916052	-0.216616325413051\\
0.0761999304565168	-0.230543653812325\\
0.0797134125340549	-0.237215968737979\\
0.0834952714091901	-0.235523536759573\\
0.0874282917341887	-0.224579715384135\\
0.0913552159090257	-0.203917076839757\\
0.0950809936254724	-0.173722426895825\\
0.0983845338183291	-0.135058545816367\\
0.101042068231994	-0.0899814241083577\\
0.102860450070868	-0.0414472582528619\\
0.103713106798486	0.00704939167430759\\
0.10356688393617	0.0520468933216796\\
0.102488817024193	0.0907076972377363\\
0.100629512107216	0.121202899629085\\
0.0981901943760853	0.142770768231165\\
0.0953864564793171	0.155503110132655\\
0.0924198783499037	0.16001959093518\\
0.0894620801243376	0.157176564454926\\
0.0866498116437506	0.147877721081047\\
0.0840870963081546	0.132984053058287\\
0.0818506477327625	0.113290459181726\\
0.0799960682395692	0.0895361528237253\\
0.0785635602013302	0.062427105690318\\
0.0775826610661109	0.0326595306914163\\
0.0770758984060689	0.000940347728925234\\
0.0770614027768647	-0.0319959390613806\\
0.0775545331663113	-0.0653732060391674\\
0.0785685308615018	-0.0983615992297748\\
0.0801141561304361	-0.130066859326896\\
0.0821981954655382	-0.159521439118844\\
0.0848206703522651	-0.185679071232642\\
0.0879705550107328	-0.207416348365077\\
0.0916198612675279	-0.223547508063746\\
0.0957161396026283	-0.232861844601225\\
0.100173866967824	-0.234196228144689\\
0.10486593532833	-0.226556080640565\\
0.109617538388469	-0.209293050026847\\
0.11420597617762	-0.182331386493422\\
0.11837064297408	-0.146404158805109\\
0.121836633856448	-0.103220132989246\\
0.124351832566373	-0.0554545501748308\\
0.12573100755405	-0.00647942339183698\\
};
\addplot [color=mycolor5, forget plot]
  table[row sep=crcr]{%
0.3	0.45\\
0.273	0.347\\
0.2540538	0.293345528892062\\
0.239148690734392	0.259217499744296\\
0.226750385598534	0.233201714927442\\
0.216174669842127	0.209926638945089\\
0.207098505469122	0.186675046530236\\
0.199366480840164	0.162096853942705\\
0.192903144975	0.135615083012445\\
0.187671029771171	0.107117834666952\\
0.18365044690341	0.076788553880487\\
0.180829996455966	0.0450104058037743\\
0.179202150151571	0.0123127916516569\\
0.178760854403902	-0.0206574122672503\\
0.179499401737235	-0.0531499703414367\\
0.181407479312963	-0.0843170540006954\\
0.184466628158835	-0.113215279558584\\
0.188643516334082	-0.138808454865959\\
0.193880579338645	-0.159976513127672\\
0.200083847147799	-0.175538189410903\\
0.207108318399537	-0.184297690323214\\
0.214742235345089	-0.185127817655083\\
0.222693187622651	-0.177101373210794\\
0.230581041489183	-0.159675249462108\\
0.237944658553387	-0.132912776693906\\
0.24426983560695	-0.097698062559729\\
0.249042773543066	-0.0558597548283025\\
0.251825067197441	-0.0101056345706097\\
0.252334037618405	0.0362924594944133\\
0.250502473050539	0.080019897782743\\
0.246493436592974	0.118280918257683\\
0.240662342588032	0.149188901545795\\
0.233481512481202	0.171858967378013\\
0.225456334153827	0.186234941307127\\
0.217058764722136	0.192792313020177\\
0.208689312459719	0.192259076018305\\
0.200664829582039	0.185427198434226\\
0.193223086147303	0.173059300872565\\
0.186535275707085	0.155862030633429\\
0.18072052233579	0.134494509415209\\
0.175859338737228	0.109589850773829\\
0.172004858999348	0.0817782061757992\\
0.169191609234851	0.0517068861931922\\
0.167441934978141	0.020056795543897\\
0.166770265247075	-0.0124441156320343\\
0.167185326940019	-0.0450094213401396\\
0.168690309904445	-0.0767842026612872\\
0.171280860092985	-0.106836584579438\\
0.174940672512217	-0.134152115480312\\
0.17963440477243	-0.157633953429288\\
0.185297701051669	-0.176114157782578\\
0.191824410763621	-0.188384333705465\\
0.199051753525651	-0.19325689152703\\
0.206745378153525	-0.189669696425735\\
0.214588044775886	-0.176843773898502\\
0.22217775671022	-0.154491396796087\\
0.229042667104456	-0.123046278095712\\
0.23467923664692	-0.0838508252360507\\
0.238614846178642	-0.0392036324059163\\
0.240485759921878	0.0078242146095097\\
0.240109437482646	0.0538231884484879\\
};
\addplot [color=mycolor6, forget plot]
  table[row sep=crcr]{%
0.5	0.45\\
0.455	0.379\\
0.420511	0.336151928892062\\
0.392239883245934	0.305141245318852\\
0.368302169958457	0.278905730364427\\
0.347757852817044	0.253973854866075\\
0.330093572329065	0.228580216443238\\
0.315003000287165	0.201874848037322\\
0.302284763724311	0.173542231904816\\
0.291792929210803	0.143602262582959\\
0.283412504242727	0.112301684111619\\
0.277046963937777	0.0800538676724423\\
0.272611227739752	0.0474055936288878\\
0.270026568323571	0.0150192493010801\\
0.269215449054058	-0.0163363998472897\\
0.270095051298221	-0.0457915964933772\\
0.272568668019002	-0.0723813828013312\\
0.276514447438909	-0.0950567351041652\\
0.281771359555444	-0.112707562939282\\
0.28812291220376	-0.124206149252473\\
0.295280239691008	-0.128481209495195\\
0.302867832158114	-0.12463203312628\\
0.310417238896197	-0.112086140289304\\
0.317375932933624	-0.0907891044275579\\
0.323138788277205	-0.0613904779208045\\
0.327106317206623	-0.0253612141785572\\
0.32876547988059	0.0150365700077673\\
0.327776778849718	0.0569819983484162\\
0.324041303675505	0.0975089865721725\\
0.31772191584972	0.133982835495862\\
0.309208059212775	0.164471127322994\\
0.299036899597559	0.187887040278181\\
0.287799867997689	0.203900734576092\\
0.276063347098563	0.212715382578413\\
0.264318762999774	0.214826996388717\\
0.252962201810889	0.210844097972252\\
0.242295084358511	0.201386272744076\\
0.232536103569876	0.187045393162833\\
0.22383714218652	0.168383897160302\\
0.216299028120402	0.145949468372644\\
0.209985282487664	0.120293900712111\\
0.204933292743148	0.0919906336321284\\
0.201162904052796	0.0616493840810046\\
0.198682590225836	0.0299281187293336\\
0.19749335099589	-0.00245682996703527\\
0.197590392512493	-0.0347218043445885\\
0.198962531502331	-0.0660078054060341\\
0.201589147514831	-0.0953740197564805\\
0.205434420982385	-0.121794712485597\\
0.210438586230024	-0.14416308330909\\
0.216506081317649	-0.161308098373285\\
0.22349091817037	-0.172033199543268\\
0.231180489714712	-0.175188372833085\\
0.239280516479487	-0.169787413763244\\
0.247405880490884	-0.155176913784813\\
0.255084216688242	-0.131248042239035\\
0.261780077497524	-0.0986539768077664\\
0.266945206636359	-0.0589602193163486\\
0.270093036222105	-0.0146322064511911\\
0.270883447635511	0.0312131253843286\\
0.269192423832394	0.0752779890190817\\
};
\addplot [color=mycolor7, forget plot]
  table[row sep=crcr]{%
0.5	-0.45\\
0.545	-0.341\\
0.582169	-0.210904071107938\\
0.606725362434567	-0.0726782691579669\\
0.615544512273764	0.0546344461802627\\
0.608818525568289	0.153127932645868\\
0.590173101132933	0.214270195322572\\
0.564881800002157	0.240830215200634\\
0.537673678910669	0.241781141878821\\
0.511673807701627	0.226400286159369\\
0.488505188404847	0.201619305461199\\
0.468806773044771	0.171955428627431\\
0.4526839991243	0.14026777735593\\
0.439984603443948	0.108459076974567\\
0.430440538649438	0.0779347656134816\\
0.423731282151401	0.0498629758434247\\
0.419505581614197	0.0253058691954667\\
0.417382390939178	0.00527078836143363\\
0.416942404089492	-0.00929277937058561\\
0.417717314843781	-0.017538492946963\\
0.419182541279823	-0.0188046684231718\\
0.420759059019333	-0.0127245614207581\\
0.421829853917299	0.000650181160385252\\
0.421775000752518	0.0207414174663488\\
0.420025358479023	0.0463922553188784\\
0.416128173744831	0.0759121172387325\\
0.409810339602499	0.107240274643522\\
0.401020704928354	0.138206765566612\\
0.389935950017676	0.166823133387315\\
0.376925882617214	0.191514826525337\\
0.362488503612745	0.211231289753077\\
0.347174720784988	0.225425666076962\\
0.331522302249381	0.23394700269266\\
0.316010572461978	0.236905775442232\\
0.301037626518568	0.234559047758107\\
0.286915406715457	0.227234664485888\\
0.273875981479293	0.215292697330409\\
0.262083281721955	0.199112956027668\\
0.251646446332137	0.179096872809253\\
0.242632628013808	0.155675401226813\\
0.235078241670455	0.129318278226108\\
0.228998258978206	0.100542689998484\\
0.224393438785678	0.0699209164677531\\
0.221255479807829	0.0380872097185045\\
0.219570079035667	0.00574430866014476\\
0.219317823374365	-0.0263301986862143\\
0.22047275974734	-0.0572779222584138\\
0.222998404065921	-0.0861551754701139\\
0.226840897392292	-0.111932220538776\\
0.231919058463117	-0.133500339472676\\
0.238111313070119	-0.149693231503422\\
0.245240043452317	-0.15933198759734\\
0.25305496016466	-0.161305009805421\\
0.261218766730794	-0.154693580618241\\
0.269300540000848	-0.13894676676851\\
0.276784227865274	-0.114091834166794\\
0.283099991910392	-0.0809370593244548\\
0.287682648078393	-0.0411906062135297\\
0.290052612612685	0.00259565313155507\\
0.289902037418237	0.0473144126442598\\
0.287158728493273	0.0897863308672501\\
};
\addplot [color=mycolor1, forget plot]
  table[row sep=crcr]{%
0.3	-0.45\\
0.327	-0.373\\
0.3513942	-0.283662471107938\\
0.371329669420999	-0.188367109433503\\
0.385318928716146	-0.0960426292893898\\
0.39272033732192	-0.0159943072550896\\
0.393976595270009	0.0449631950127557\\
0.390433705973292	0.084323279695846\\
0.383849175854998	0.103600510454231\\
0.375895781743795	0.106595243902682\\
0.367882041236401	0.0976305997249779\\
0.360698732373609	0.0805518836044139\\
0.354887739912326	0.058441723747062\\
0.350739689660891	0.0337065807787919\\
0.348375242524514	0.0082700395001508\\
0.347799027121204	-0.0162493809338431\\
0.348929330897226	-0.0384119138760989\\
0.35160993957868	-0.0568841243495895\\
0.355610144283789	-0.0704156206277479\\
0.360618246086042	-0.0778620462267546\\
0.366233940995435	-0.0782539179921368\\
0.371965789152353	-0.0709093258019285\\
0.377240957818389	-0.0555800450429048\\
0.381434371703903	-0.0326041880448863\\
0.383921643300267	-0.00301775403045793\\
0.384153359517557	0.031436268850726\\
0.381738089859616	0.0684515894797652\\
0.376511974056444	0.105457650957842\\
0.368570760388145	0.140004491426763\\
0.358250448015562	0.170110265101676\\
0.346062032278618	0.194457863711892\\
0.332603135576879	0.212401047035236\\
0.318474084728132	0.223827229725798\\
0.304217450303301	0.228966726740542\\
0.290286315540641	0.228224357622457\\
0.27703623396247	0.222068046648616\\
0.264732054897083	0.210973678084747\\
0.253561755831369	0.195409163234574\\
0.24365209772431	0.175839513340022\\
0.235083364466686	0.152740057965559\\
0.22790203512361	0.126610844927446\\
0.222131061278073	0.097989419352926\\
0.217777762535096	0.0674614635880043\\
0.214839441215588	0.0356697493077535\\
0.213306787413672	0.00332206177429603\\
0.213165063748739	-0.0288014586100704\\
0.214392956700874	-0.0598416207089166\\
0.216958881100385	-0.0888550174883808\\
0.220814458135272	-0.114812520527943\\
0.225884911036776	-0.136605370175738\\
0.232056329414635	-0.153065222003349\\
0.239160280130462	-0.163007310926997\\
0.246957254959384	-0.165308156790802\\
0.255122064684075	-0.159028875829217\\
0.263236419713262	-0.143588629963234\\
0.270795971085873	-0.11897579462255\\
0.277239604253978	-0.0859553550897959\\
0.282005649979699	-0.0461970104749639\\
0.284611213572922	-0.00223129324921138\\
0.284738223788821	0.0428281639994941\\
0.28229926071975	0.0857656257137926\\
};
\addplot [color=mycolor2, forget plot]
  table[row sep=crcr]{%
0.1	-0.45\\
0.109	-0.389\\
0.1174802	-0.320348871107938\\
0.125007129889507	-0.24851508739585\\
0.131220361451426	-0.179123832131901\\
0.135921300250808	-0.117854112481321\\
0.139125077092482	-0.0691290914507793\\
0.141048595127966	-0.035225760535516\\
0.142042303935136	-0.0161996517280084\\
0.142502511106015	-0.0104894657624744\\
0.142801466148277	-0.0157342957671227\\
0.143250842249148	-0.0294153539195507\\
0.144093597093955	-0.0491992834558748\\
0.145511457439475	-0.0730447473940957\\
0.147637226969798	-0.0991797327832388\\
0.150565751113743	-0.126030200860539\\
0.154360917484859	-0.152143221048344\\
0.159057910922884	-0.176122880943445\\
0.164660658424601	-0.196586038255027\\
0.171134655723832	-0.212142553108887\\
0.178395644281965	-0.22140679914071\\
0.186295245998188	-0.223051007458287\\
0.194605914459105	-0.215913364068622\\
0.203009517990808	-0.199170474678698\\
0.211096218403313	-0.172568984404234\\
0.218381950407599	-0.136679937417902\\
0.224351636670581	-0.0930952675254199\\
0.228528851797704	-0.0444503406142151\\
0.230560488858221	0.00582912541426722\\
0.230291695657195	0.0540596305801966\\
0.227801798858612	0.0969615558655694\\
0.223384195489351	0.132195873611269\\
0.217478101714618	0.158573603899122\\
0.210580844443012	0.175919238860864\\
0.203171800068394	0.18473684765533\\
0.195665136492975	0.18586402755771\\
0.188391714428732	0.180225129195107\\
0.181601130214291	0.16870351781842\\
0.175473780312901	0.152101623962753\\
0.170135810923206	0.131148539209933\\
0.16567319830923	0.106525865493018\\
0.162143502141453	0.0788961084239711\\
0.159585003876414	0.0489274083775455\\
0.158023387745295	0.017313412172432\\
0.157476202936312	-0.0152108961095762\\
0.157955273768831	-0.0478561652366785\\
0.159467100505128	-0.0797675566017637\\
0.16201116069826	-0.110015420357403\\
0.165575905887622	-0.13758882223609\\
0.170132184663973	-0.161394648058202\\
0.175623869477416	-0.180267332968825\\
0.181955718788687	-0.192997218657116\\
0.188979108317682	-0.198388651370824\\
0.196477370404963	-0.195360736089005\\
0.204154163146392	-0.183101153875641\\
0.211630335714516	-0.161272046977515\\
0.21845634720316	-0.130242342138209\\
0.224146800466099	-0.0912842224672185\\
0.228239013745912	-0.04663760073312\\
0.230367917744873	0.000652984060187142\\
0.230337832429219	0.04715771614764\\
};
\addplot [color=mycolor3, forget plot]
  table[row sep=crcr]{%
-0.1	-0.45\\
-0.109	-0.389\\
-0.1174802	-0.320348871107938\\
-0.125007129889507	-0.24851508739585\\
-0.131220361451426	-0.179123832131901\\
-0.135921300250809	-0.117854112481321\\
-0.139125077092482	-0.0691290914507793\\
-0.141048595127966	-0.035225760535516\\
-0.142042303935136	-0.0161996517280084\\
-0.142502511106015	-0.0104894657624743\\
-0.142801466148277	-0.0157342957671227\\
-0.143250842249149	-0.0294153539195507\\
-0.144093597093955	-0.0491992834558747\\
-0.145511457439475	-0.0730447473940956\\
-0.147637226969798	-0.0991797327832387\\
-0.150565751113743	-0.126030200860539\\
-0.154360917484859	-0.152143221048344\\
-0.159057910922884	-0.176122880943445\\
-0.164660658424601	-0.196586038255026\\
-0.171134655723832	-0.212142553108887\\
-0.178395644281965	-0.22140679914071\\
-0.186295245998188	-0.223051007458287\\
-0.194605914459105	-0.215913364068622\\
-0.203009517990808	-0.199170474678698\\
-0.211096218403313	-0.172568984404234\\
-0.2183819504076	-0.136679937417902\\
-0.224351636670582	-0.0930952675254198\\
-0.228528851797704	-0.044450340614215\\
-0.230560488858221	0.00582912541426733\\
-0.230291695657195	0.0540596305801967\\
-0.227801798858612	0.0969615558655695\\
-0.223384195489351	0.132195873611269\\
-0.217478101714618	0.158573603899122\\
-0.210580844443012	0.175919238860864\\
-0.203171800068394	0.18473684765533\\
-0.195665136492975	0.18586402755771\\
-0.188391714428732	0.180225129195107\\
-0.181601130214291	0.16870351781842\\
-0.175473780312901	0.152101623962753\\
-0.170135810923206	0.131148539209933\\
-0.165673198309231	0.106525865493018\\
-0.162143502141453	0.0788961084239711\\
-0.159585003876414	0.0489274083775455\\
-0.158023387745295	0.017313412172432\\
-0.157476202936312	-0.0152108961095762\\
-0.157955273768831	-0.0478561652366785\\
-0.159467100505128	-0.0797675566017637\\
-0.16201116069826	-0.110015420357403\\
-0.165575905887622	-0.13758882223609\\
-0.170132184663973	-0.161394648058202\\
-0.175623869477416	-0.180267332968825\\
-0.181955718788687	-0.192997218657116\\
-0.188979108317682	-0.198388651370824\\
-0.196477370404963	-0.195360736089005\\
-0.204154163146392	-0.183101153875641\\
-0.211630335714516	-0.161272046977515\\
-0.21845634720316	-0.130242342138209\\
-0.224146800466099	-0.0912842224672184\\
-0.228239013745912	-0.0466376007331199\\
-0.230367917744873	0.000652984060187191\\
-0.230337832429219	0.0471577161476401\\
};
\addplot [color=mycolor4, forget plot]
  table[row sep=crcr]{%
-0.3	-0.45\\
-0.327	-0.373\\
-0.3513942	-0.283662471107938\\
-0.371329669420999	-0.188367109433503\\
-0.385318928716146	-0.0960426292893898\\
-0.39272033732192	-0.0159943072550896\\
-0.39397659527001	0.0449631950127558\\
-0.390433705973292	0.084323279695846\\
-0.383849175854998	0.103600510454231\\
-0.375895781743795	0.106595243902682\\
-0.367882041236401	0.097630599724978\\
-0.36069873237361	0.080551883604414\\
-0.354887739912326	0.058441723747062\\
-0.350739689660891	0.0337065807787919\\
-0.348375242524514	0.00827003950015082\\
-0.347799027121204	-0.016249380933843\\
-0.348929330897226	-0.0384119138760988\\
-0.35160993957868	-0.0568841243495895\\
-0.355610144283789	-0.0704156206277479\\
-0.360618246086042	-0.0778620462267546\\
-0.366233940995435	-0.0782539179921367\\
-0.371965789152353	-0.0709093258019285\\
-0.377240957818389	-0.0555800450429048\\
-0.381434371703903	-0.0326041880448863\\
-0.383921643300267	-0.00301775403045788\\
-0.384153359517557	0.0314362688507261\\
-0.381738089859616	0.0684515894797652\\
-0.376511974056444	0.105457650957842\\
-0.368570760388145	0.140004491426763\\
-0.358250448015562	0.170110265101676\\
-0.346062032278618	0.194457863711892\\
-0.332603135576879	0.212401047035236\\
-0.318474084728132	0.223827229725798\\
-0.304217450303301	0.228966726740542\\
-0.290286315540641	0.228224357622457\\
-0.27703623396247	0.222068046648616\\
-0.264732054897083	0.210973678084747\\
-0.253561755831369	0.195409163234574\\
-0.24365209772431	0.175839513340022\\
-0.235083364466686	0.152740057965559\\
-0.22790203512361	0.126610844927446\\
-0.222131061278073	0.097989419352926\\
-0.217777762535096	0.0674614635880043\\
-0.214839441215588	0.0356697493077535\\
-0.213306787413672	0.00332206177429603\\
-0.213165063748739	-0.0288014586100704\\
-0.214392956700874	-0.0598416207089166\\
-0.216958881100385	-0.0888550174883808\\
-0.220814458135272	-0.114812520527943\\
-0.225884911036776	-0.136605370175738\\
-0.232056329414635	-0.153065222003349\\
-0.239160280130462	-0.163007310926997\\
-0.246957254959384	-0.165308156790802\\
-0.255122064684075	-0.159028875829217\\
-0.263236419713262	-0.143588629963234\\
-0.270795971085873	-0.11897579462255\\
-0.277239604253978	-0.0859553550897959\\
-0.282005649979699	-0.0461970104749639\\
-0.284611213572922	-0.00223129324921138\\
-0.284738223788821	0.0428281639994941\\
-0.28229926071975	0.0857656257137926\\
};
\addplot [color=mycolor5, forget plot]
  table[row sep=crcr]{%
-0.5	-0.45\\
-0.545	-0.341\\
-0.582169	-0.210904071107938\\
-0.606725362434567	-0.0726782691579669\\
-0.615544512273764	0.0546344461802627\\
-0.608818525568289	0.153127932645868\\
-0.590173101132933	0.214270195322572\\
-0.564881800002157	0.240830215200634\\
-0.537673678910669	0.241781141878821\\
-0.511673807701627	0.226400286159369\\
-0.488505188404847	0.201619305461199\\
-0.468806773044771	0.171955428627431\\
-0.4526839991243	0.14026777735593\\
-0.439984603443948	0.108459076974567\\
-0.430440538649438	0.0779347656134816\\
-0.423731282151401	0.0498629758434247\\
-0.419505581614197	0.0253058691954667\\
-0.417382390939178	0.00527078836143363\\
-0.416942404089492	-0.00929277937058561\\
-0.417717314843781	-0.017538492946963\\
-0.419182541279823	-0.0188046684231718\\
-0.420759059019333	-0.0127245614207581\\
-0.421829853917299	0.000650181160385252\\
-0.421775000752518	0.0207414174663488\\
-0.420025358479023	0.0463922553188784\\
-0.416128173744831	0.0759121172387325\\
-0.409810339602499	0.107240274643522\\
-0.401020704928354	0.138206765566612\\
-0.389935950017676	0.166823133387315\\
-0.376925882617214	0.191514826525337\\
-0.362488503612745	0.211231289753077\\
-0.347174720784988	0.225425666076962\\
-0.331522302249381	0.23394700269266\\
-0.316010572461978	0.236905775442232\\
-0.301037626518568	0.234559047758107\\
-0.286915406715457	0.227234664485888\\
-0.273875981479293	0.215292697330409\\
-0.262083281721955	0.199112956027668\\
-0.251646446332137	0.179096872809253\\
-0.242632628013808	0.155675401226813\\
-0.235078241670455	0.129318278226108\\
-0.228998258978206	0.100542689998484\\
-0.224393438785678	0.0699209164677531\\
-0.221255479807829	0.0380872097185045\\
-0.219570079035667	0.00574430866014476\\
-0.219317823374365	-0.0263301986862143\\
-0.22047275974734	-0.0572779222584138\\
-0.222998404065921	-0.0861551754701139\\
-0.226840897392292	-0.111932220538776\\
-0.231919058463117	-0.133500339472676\\
-0.238111313070119	-0.149693231503422\\
-0.245240043452317	-0.15933198759734\\
-0.25305496016466	-0.161305009805421\\
-0.261218766730794	-0.154693580618241\\
-0.269300540000848	-0.13894676676851\\
-0.276784227865274	-0.114091834166794\\
-0.283099991910392	-0.0809370593244548\\
-0.287682648078393	-0.0411906062135297\\
-0.290052612612685	0.00259565313155507\\
-0.289902037418237	0.0473144126442598\\
-0.287158728493273	0.0897863308672501\\
};
\end{axis}
\end{tikzpicture}%

%% file: stseeg.ltx
\tikzsetnextfilename{stseeg}
\begin{tikzpicture}
   \begin{axis}[small,xlabel={decay rate $\mu$},ylabel={lower bound $r$}]
   \addplot+[no marks,domain=0:0.62] ({1-sqrt(16*x^2*(2*0+x^4)*sqrt(1+x^4))},{x^2});
   \addlegendentry{$\sigma=0$}
   \addplot+[no marks,domain=0:0.34] ({1-sqrt(16*x^2*(2/4+x^4)*sqrt(1+x^4))},{x^2});
   \addlegendentry{$\sigma=1/4$}
   \end{axis}
\end{tikzpicture}